\documentclass[twoside, 11pt]{article}
\usepackage{amsmath,amssymb,amsthm,mathrsfs}
\usepackage[margin=2.5cm]{geometry} 
\usepackage{lipsum}
\usepackage{titlesec,hyperref}
\usepackage{fancyhdr}
\usepackage[numbers,sort&compress]{natbib}
\usepackage{color}

\fancypagestyle{plain}
{
\fancyhf{}

}

\titleformat{\subsection}{\it}{\thesubsection.\enspace}{1.5pt}{}
\titleformat{\subsubsection}{\it}{\thesubsubsection.\enspace}{1.5pt}{}

\newtheorem{theo}{Theorem}[section]
\newtheorem{lemm}[theo]{Lemma}

\newtheorem{prop}[theo]{Proposition}
\newtheorem{rema}{Remark}[section]
\numberwithin{equation}{section}

\allowdisplaybreaks

\def\vr{\varrho}
\def\n{\nabla}
\def\nc{\nabla\cdot}
\def\nt{\nabla\times}
\def\p{\partial}
\def\d{\delta}
\def\x{|x|}
\def\g{\gamma}
\def\var{\varepsilon}
\def\f{\widehat}
\def\c{\mathcal}
\def\bq{\begin{equation}}
\def\eq{\end{equation}}
\def\bqq{\begin{equation*}}
\def\eqq{\end{equation*}}

\begin{document}
\title{Lower Bound and Space-time Decay Rates of Higher Order Derivatives of Solution
       for the Compressible Navier-Stokes and Hall-MHD Equations
       \hspace{-510mm}}
\author{Jincheng Gao \quad Zeyu Lyu \quad  Zheng-an Yao \\[10pt]
\small {School of Mathematics, Sun Yat-sen University,}\\
\small {510275, Guangzhou, P.R. China}\\[5pt]
}

\footnotetext{Email: \it gaojc1998@163.com(J.C.Gao), lvzy3@mail2.sysu.edu.cn(Z.Y.Lyu), mcsyao@mail.sysu.edu.cn(Z.A.Yao).}
\date{}

\maketitle

\begin{abstract}
In this paper, we address the lower bound and space-time decay rates for the compressible
Navier-Stokes and Hall-MHD equations under $H^3-$framework in $\mathbb{R}^3$.
First of all, the lower bound of decay rate for the density, velocity and magnetic field
converging to the equilibrium status in $L^2$ is $(1+t)^{-\frac{3}{4}}$;
the lower bound of decay rate for the first order spatial derivative of density and velocity
converging to zero in $L^2$ is $(1+t)^{-\frac{5}{4}}$,
and the $k(\in [1, 3])-$th order spatial derivative of magnetic field
converging to zero in $L^2$ is $(1+t)^{-\frac{3+2k}{4}}$.
Secondly, the lower bound of decay rate for time derivatives of density and velocity
converging to zero in $L^2$ is $(1+t)^{-\frac{5}{4}}$;
however, the lower bound of decay rate for time derivatives of magnetic field converging
to zero in $L^2$ is $(1+t)^{-\frac{7}{4}}$.
Finally, we address the decay rate of solution in weighted Sobolev space $H^3_\g$.
More precisely, the upper bound of decay rate of the $k(\in [0, 2])$-th order spatial derivatives
of density and velocity converging to the $k(\in [0, 2])$-th order derivatives of constant equilibrium
in weighted space $L^2_\g$ is $t^{-\frac{3}{4}+\g-\frac{k}{2}}$;
however, the upper bounds of decay rate of the $k(\in [0, 3])$-th order spatial derivatives
of magnetic field converging to zero in weighted space $L^2_\g$ is $t^{-\frac{3}{4}+\frac{\g}{2}-\frac{k}{2}}$.


\end{abstract}


\section{Introduction}

The application of Hall-magnetohydrodynamics(in short, Hall-MHD) system
covers a very wide range of physical objects, for example, magnetic reconnection
in space plasmas, star formulation and neutron stars, refer to
\cite{{Homann},{Wardle},{Balbus-Terquem}} and the references therein.
Recently, Acheritogaray et al.\cite{Liu} derived the Hall-MHD equations
from the two-fluid Euler-Maxwell system for
electrons and ions through a set of scaling limits or from the kinetic
equations by taking macroscopic quantities in the equations under some closure assumptions.
They also established the global existence of weak solutions with periodic boundary condition.
In this paper, we investigate the following compressible Hall-MHD equations in
three-dimensional whole space $\mathbb{R}^3$(see \cite{Liu}):
\begin{equation}\label{eq-MHD}
\left\{
\begin{aligned}
&\rho_t+{\rm div}(\rho u)=0,\\
&(\rho u)_t+{\rm div}(\rho u\otimes u)-\mu \Delta u-(\mu+\nu)\nabla {\rm div}u
  +\nabla P(\rho)=(\nabla \times B)\times B,\\
&B_t -{\nabla \times}(u \times B)+{{\nabla \times} { \left[\frac{({\nabla \times} B)\times B}{\rho}\right]}}= \Delta B, \ {\rm div} B=0,
\end{aligned}
\right.
\end{equation}
where the functions $\rho, u,$ and $B$
represent density, velocity, and magnetic field respectively.
The pressure $P(\rho)$ is a smooth function in a neighborhood of $1$ with $P'(1)>0$.
For the sake of simplicity, we assume $P'(1)=1$.
The constants $\mu$ and $\nu$ denote the viscosity coefficients of the flow
and satisfy physical condition:
$
\mu>0, \ 2\mu+3\nu \ge 0.
$
To complete the system \eqref{eq-MHD}, the initial data are given by
\begin{equation}\label{ID1}
\left.(\rho, u, B)(x,t)\right|_{t=0}=(\rho_0(x), u_0(x), B_0(x)).
\end{equation}
Furthermore, as the space variable tends to infinity, we assume
\begin{equation}\label{BC1}
\underset{|x|\rightarrow \infty}{\lim}(\rho_0-1, u_0, B_0)(x)=0.
\end{equation}
If the Hall effect term ${\nabla \times}\left[\frac{(\nabla \times B)\times B}{\rho}\right]$ is neglected,
the compressible Hall-MHD equations transform into the well-known compressible MHD equations,
which can be obtained as the singular limit of the full coupled Navier-Stokes equations and Maxwell's
equations when the dielectric constant vanishes \cite{Kawashima-Shizuta}.
In the sequence, we will describe some mathematical results related to the
Navier-Stokes and Hall-MHD equations.

(I)\textbf{Some results for the incompressible Hall-MHD equations}.
For the incompressible Hall-MHD equations(i.e., $\rho=$constant),
Chae et al.\cite{Chae-Degond-Liu} proved local existence of smooth solutions for large data
and global smooth solutions for small data in three-dimensional whole space.
They also showed a Liouville theorem for the stationary solutions.
Chae and Lee \cite{Chae-Lee} established an optimal blow-up criterion for classical solutions
and proved two global-in-time existence results of classical solutions for small initial data,
the smallness condition of which are given by the suitable Sobolev and Besov norms respectively.
Later, Fan et al.\cite{Fan-Li-Nakamura} also established some new regularity criteria,
which were also built for density-dependent incompressible Hall-MHD equations with positive initial
density by Fan and Ozawa \cite{Fan-Ozawa}.
On one hand, Maicon and Lucas \cite{Maicon-Lucas} proved a stability theorem for
global large solutions under a suitable integrable hypothesis and constructed a special large
solution by assuming the condition of curl-free magnetic fields.
On the other hand, Fan et al. \cite{Fan-Huang-Nakamura} established the global well-posedness of the axisymmetric solutions.
Recently, Chae and Schonbek \cite{Chae-Schonbek} established temporal decay estimates for weak solutions
and obtained algebraic time decay for higher order Sobolev norms of small initial data solution.
Furthermore, Weng \cite{Weng-JDE} extended this result
by providing upper and lower bounds on the decay of higher order derivatives.
In \cite{Chae-Weng}, Chae and Weng have showed that the incompressible Hall-MHD without resistivity
is not globally in time well-posed in any $H^m(\mathbb{R}^3)$ with $m >7/2$,
i.e., for some axisymmetric smooth data, either the solution will
become singular instantaneously, or the solution blows up in finite time.

(II)\textbf{Lower and upper bounds of decay rate for the incompressible Navier-Stokes flows}.
First of all, the problem of determining whether or not weak solutions with large initial data
decay to zero in $L^2$ as time tends to infinity was posed by Leray \cite{Leray}.
This was answered affirmatively along this direction by Kato \cite{Kato} and Masuda \cite{Masuda}.
The algebraic decay rate for the weak solution of incompressible Navier-Stokes equation
was firstly obtained by Schonbek \cite{Schonbek}.
This was improved to the optimal one in \cite{Schonbek-CPDE}.
Furthermore, Schonbek also addressed the lower bound of decay rate for solution
of Navier-Stokes equation \cite{{Schonbek-CPDE},{Schonbek-JAMS}}
and MHD equation \cite{Schonbek-Schonbek-Suli}.
The upper bound of decay rates of solution for the higher order spatial derivatives were studied
in \cite{{Schonbek-CPDE-1995},{Schonbek-Wiegner}}. Specifically, if the decay rate
$\|u(t)\|_{L^2}=\mathcal{O}(t^{-\theta})$ holds on for solution of Navier-Stokes equation
on $\mathbb{R}^n(n \le 5)$, then we have
\begin{equation*}
\|\nabla^k u(t)\|_{L^2}=\mathcal{O}(t^{-\theta-\frac{k}{2}}),k \in\mathbb{N},
\end{equation*}
which implies the higher order spatial derivatives admit the optimal decay rates
in the sense that they coincide with the rates for solution to the heat system.
However, this property of decay rate for the 3D exterior domain is still an open question
(see Remark (3) on page 401 in \cite{Han-CMP}).
Based on so-called Gevrey estimates, Oliver and Titi \cite{Oliver-Titi} established
the lower and upper bounds of decay rate for the higher order derivatives
of solution to the incompressible Navier-Stokes equation in whole space.
More precisely, for real number $\theta$ and small constant $\var$, we assume
$$\|u(t)\|_{L^2}\lesssim (1+t)^{-\theta},
\|(u-v)(t)\|_{L^2}\lesssim \var (1+t)^{-\theta},$$
and
$$(1+t)^{-\theta-k}\lesssim \|\n^k v(t)\|_{L^2}\lesssim (1+t)^{-\theta-k},k\in \mathbb{N},$$
where $u(t)$ is a solution to the incompressible Navier-Stokes equations,
and $v(t)$ solves the heat equation. Then, they \cite{Oliver-Titi} established the following decay rate
\begin{equation}\label{Decay-Titi}
(1+t)^{-\theta-k} \lesssim \|\n^k u(t)\|_{L^2}\lesssim (1+t)^{-\theta-k},
\end{equation}
for every real number $k>0$. Later, this result was generalized to the incompressible Hall-MHD equations
in three dimensional whole space by Weng \cite{Weng-JDE}.
The equation type also changes from parabolic to hyperbolic-parabolic coupling
when the fluid changes from incompressibility to compressibility.
Thus, a natural question is: whether the solution of compressible Navier-Stokes equation or Hall-MHD equation
obeys the lower bound and upper bound of decay rate for the higher order derivative like \eqref{Decay-Titi}.
\textit{The first purpose of this paper is to provide an affirmative answer along this direction}.

(III)\textbf{Lower bound of decay rate for the compressible Navier-Stokes flows and related models}.
In order to answer the question mentioned above, we will review some results of lower bound of decay rate
for the compressible Navier-Stokes equations and related models.

(1)\textbf{Compressible Navier-Stokes equations}.
   When there is no external or internal force involved, there are many results on the problem of long time
   behavior of global smooth solutions to the compressible Navier-Stokes equations.
   For multi-dimensional Navier-Stokes equations, the $H^s(s \ge 3)$ global existence and decay rate of strong solutions
   are obtained in whole space first by Matsumura and Nishida \cite{{Matsumura-Nishida-1980},{Matsumura-Nishida-1979}}
   and the optimal $L^p(p\ge 2)$ decay rate is established by Ponce \cite{Ponce-1985}.
   The long time decay rate of global solution in multi-dimensional half space is also investigated for the
   compressible Navier-Stokes equations by Kagei and Kobayashi \cite{{Kagei-Kobayashi-2002}}.
   Therein, assume the initial data belongs to $L^1$ and the lower frequency of initial data satisfies some condition additionally, the optimal $L^2$ time-decay rate in three dimension is established as
   $$
   (1+t)^{-\frac{3}{4}}\lesssim \|(\rho-\bar{\rho})(t)\|_{L^2}+\|m(t)\|_{L^2} \lesssim (1+t)^{-\frac{3}{4}},
   $$
   where $(\bar{\rho}, 0)$ represent the constant state and $m$ represents the momentum.
   If the initial data belongs to $\dot{B}_{1,\infty}^{-s}(s\in [0,1])$ rather than $L^1$, Li and Zhang \cite{Li-Zhang}
   established the optimal decay rate as
   \begin{equation*}
   (1+t)^{-\frac{3}{4}-\frac{s}{2}}\lesssim
   \|(\rho-\bar{\rho})(t)\|_{L^2}+\|m(t)\|_{L^2}
   \lesssim(1+t)^{-\frac{3}{4}-\frac{s}{2}}.
   \end{equation*}
 For more result about the long time behavior of compressible Navier-Stokes equation, the reader can refer to
 \cite{{Hoff-Zumbrun-1995},{Liu-Wang-1998},{Duan-Ukai-Yang-Zhao-2007},
{Duan-Liu-Ukai-Yang-2007},{Guo-Wang-2012},{Danchin-Xu-2017}}and references therein.

(2)\textbf{Compressible Navier-Stokes-Poission equations}.
   The global existence and optimal decay rate were obtained in \cite{Li-Matsumura-Zhang}
   for the compressible Navier-Stokes-Poisson equation in $\mathbb{R}^3$.
   The influences of the electric field of the internal electrostatic potential force
   governed by the self-consistent Poisson equation on the qualitative behaviors of solutions is analyzed.
   They also addressed the lower bound of decay rate as follows
   \begin{equation*}
   \|(\rho-\bar{\rho})(t)\|_{L^2}\ge c(1+t)^{-\frac{3}{4}},
   \min\{\|m(t)\|_{L^2}, \|\n \Phi (t)\|_{L^2}\}\ge c(1+t)^{-\frac{1}{4}},
   \end{equation*}
   where $m$ and $\Phi$ represent the momentum and electrostatic potential respectively.

(3)\textbf{Compressible Viscoelastic Flows}.
    The global existence of the strong solution was obtained by Hu and Wu \cite{Hu-Wu}
    under the condition that the initial data are close to the constant equilibrium state in $H^2$-framework.
    At the same time, the lower and upper bounds of decay rate were also addressed
    if the initial data satisfies some additional condition. Specifically, they got the lower bound
    of decay rate as follows
    $$
    \min\{\|(\rho-1)(t)\|_{L^2}, \|u(t)\|_{L^2}, \|(F-I)(t)\|_{L^2}\}\ge c(1+t)^{-\frac{3}{4}},
    $$
    where $F$ is a $3\times 3$ metric and denotes the deformation gradient.

(4)\textbf{Compressible MHD and Hall-MHD equation}.
   First of all, under the $H^3$-framework,  Li $\&$ Yu \cite{Li-Yu} and Chen $\&$ Tan \cite{Chen-Tan}
   not only established the global existence of classical solutions, but also obtained
   the time decay rates for the three-dimensional compressible MHD equations
   by assuming the initial data belong to $L^1$ and $L^q( q \in \left[1, \frac{6}{5}\right))$ respectively.
   Later, these results were generalized to the compressible Hall-MHD equations by Fan et al.\cite{Fan-Zhou}.
   The first author and the third author of this paper have provided better decay rate
    for the higher order derivative
   in \cite{Gao-Yao}. We should note that these results do not provide any lower bound of decay rate for the solution.

\textit{To the authors' knowledge, there are no references concerning the lower bound decay rate for the higher
         order derivative of solution for the compressible Navier-Stokes equation and Hall-MHD equation.
         Thus, the first result of this paper is to address this issue}.

\textbf{Notation:}
In this paper, the symbol $\nabla^k $ with an integer $k \ge 0$ stands for the usual any spatial derivatives of order $k$.
For example, we define
$
\nabla^k v=
\left\{\left.\partial_x^\alpha v_i\right||\alpha|=k,~i=1,2,3\right\}
,~v=(v_1, v_2, v_3).
$
We also denote the Fourier transform $\mathcal{F}(f):=\hat{f}$.
Denote by $\Lambda^s$ the pseudo-differential operator defined by
$\Lambda^s f=\mathcal{F}^{-1}(|\xi|^s \hat{u}(\xi))$.
For any $\g \in \mathbb{R}$, denote by $L^p_{\g}(\mathbb{R}^3)(2\le p <+\infty)$ the weighted Lebesgue space
with respect to the spatial variables:
$$
L_{\g}^p(\mathbb{R}^3):= \{f(x):\mathbb{R}^3 \rightarrow \mathbb{R},\
\|f\|_{L_{\g}^p(\mathbb{R}^3)}^p := \int_{\mathbb{R}^3} |x|^{p\g} |f(x)|^p dx<+\infty \}.
$$
Then, we can define the Sobolev space as follows
$$
H^{s}_{\g}(\mathbb{R}^3)\triangleq \{f\in L^2_{\g}(\mathbb{R}^3)|
\|f\|_{H^s_{\g}(\mathbb{R}^3)}^2:= \sum_{k \le s}\|\n^k u\|_{L^2_{\g}(\mathbb{R}^3)}^2<+\infty \}.
$$
Denote $L^2(\mathbb{R}^3):=L^2_0(\mathbb{R}^3)$ and $H^s(\mathbb{R}^3):=H^s_0(\mathbb{R}^3)$
as the usual  Lebesgue space and Sobolev space.
The notation $a \lesssim b$ means that $a \le C b$ for a universal constant $C>0$ independent of
time $t$.
The notation $a \approx b$ means $a \lesssim b$ and $b \lesssim a$.
For the sake of simplicity, we write $\int f dx:=\int _{\mathbb{R}^3} f dx$
and $\|(A, B)\|_X:=\| A \|_X+\| B\|_X$.

First of all, we recall the main results obtained in \cite{Gao-Yao} $\&$ \cite{Fan-Zhou} in the following.

\begin{theo}[\cite{Gao-Yao} $\&$ \cite{Fan-Zhou} ]\label{THM1}
Assume that the initial data $(\rho_0-1,u_0,B_0)\in H^3$ and there exists
a small constant $\d_0>0$ such that
\begin{equation}\label{smallness}
\|(\rho_0-1,u_0, B_0)\|_{H^3} \le \d_0,
\end{equation}
then the solution $(\rho, u, B)$ of compressible Hall-MHD equations
\eqref{eq-MHD}-\eqref{BC1} satisfies for all $t \ge 0$
\begin{equation}\label{uniform}
\|(\rho-1,u,B)(t)\|_{H^3}^2+\int_0^t (\|\nabla \rho(s)\|_{H^2}^2+\|\nabla(u, B)(s)\|_{H^3}^2)ds
\le C\|(\rho_0-1,u_0,B_0)\|_{H^3}^2.
\end{equation}
Furthermore, if $\|(\rho_0-1, u_0, B_0)\|_{L^1}$ is finite additionally,
then it holds on for all $t \ge T_* \ge 0$
\begin{equation}\label{Decay1}
\begin{aligned}
\|\n^k(\rho-1)(t)\|_{H^{3-k}}+\|\n^k u(t)\|_{H^{3-k}}
&\le C(1+t)^{-\frac{3+2k}{4}}, k=0,1,2;\\
\|\p_t \rho (t)\|_{L^2}+\|\p_t u(t)\|_{L^2}
&\le C(1+t)^{-\frac{5}{4}},\\
\|\n^k B(t)\|_{L^{2}}
&\le C(1+t)^{-\frac{3+2k}{4}}, k=0,1,2,3;\\
\|\p_t B(t)\|_{L^2}
&\le  C(1+t)^{-\frac{7}{4}}.
\end{aligned}
\end{equation}
Here $C$ is a positive constant independent of time,
and $T_*$ is a positive constant for the case $k\ge 2$.
\end{theo}

Motivated by the work of Oliver and Titi in \cite{Oliver-Titi}, we address the lower bound of decay rate for
the higher order derivatives of solution to the compressible Hall-MHD equations \eqref{eq-MHD}-\eqref{BC1}.
Both the upper and lower bounds of decay rate will give better information on the closeness
of the compressible Navier-Stokes($\&$Hall-MHD) equation and their underlying linear counterpart solutions.

\begin{theo}\label{THM2}
Denote $\vr_0 := \rho_0-1$ and $m_0 := \rho_0 u_0$, assume that the Fourier transform
$\mathcal{F}(\vr_0, m_0, B_0)=(\widehat{\vr_0}, \widehat{m_0},\widehat{B_0})$
satisfies
\begin{equation}\label{Lower-frequency-condition}
|\widehat{\vr_0}|\ge c_0, ~|\widehat{B_0}|\ge c_0,~ \widehat{m_0}=0, ~0 \le |\xi| \ll 1,
\end{equation}
where $c_0$ is a positive constant.
Then, the global classical solution $(\rho, u, B)$ obtained in Theorem \ref{THM1}
has the  decay rates for all $t \ge t_*$
\begin{gather}\label{Decay21}
c_1(1+t)^{-\frac{3+2k}{4}}
\le \|\n^k(\rho-1)(t)\|_{L^2}\le C_1(1+t)^{-\frac{3+2k}{4}}, k=0,1;\\
\label{Decay22}
c_1(1+t)^{-\frac{3+2k}{4}}\le \|\n^k  u(t)\|_{L^{2}}\le C_1(1+t)^{-\frac{3+2k}{4}},k=0,1;\\
\label{Decay23}
c_1(1+t)^{-\frac{3+2k}{4}}\le \|\n^k B(t)\|_{L^2}
\le C_1(1+t)^{-\frac{3+2k}{4}}, k=0,1,2,3.
\end{gather}
Here $t_*$ is a positive large time, $c_1$ and $C_1$ are two positive constants independent of time.
\end{theo}

\begin{rema}
The lower bounds of decay rates \eqref{Decay21}-\eqref{Decay22} for the derivatives of
density, velocity and magnetic field to the compressible Navier-Stokes and Hall-MHD equations
are obtained for the first time.
\end{rema}

\begin{rema}
Although we only established the time decay rates under the $H^3$-framework in Theorem \ref{THM2},
the method here can be applied to the $H^N(N\ge3)$-framework.
More precisely, assume that the initial data $\|(\rho_0-1,u_0,B_0)\|_{L^1 \cap H^N }$ is small.
Under the condition \eqref{Lower-frequency-condition}, the global classical
solution $(\rho, u, B)$ of the system \eqref{eq-MHD}
has the  decay rates for all $t \ge t_*$
\begin{gather*}
c_1(1+t)^{-\frac{3+2k}{4}}\le \|\n^k(\rho-1)(t)\|_{L^2}
\le C_1(1+t)^{-\frac{3+2k}{4}}, k \in [0, N-2];\\
c_1(1+t)^{-\frac{3+2k}{4}}\le \|\n^k  u(t)\|_{L^2}\le C_1(1+t)^{-\frac{3+2k}{4}}, k \in [0, N-2];\\
c_1(1+t)^{-\frac{3+2k}{4}}\le \|\n^k B(t)\|_{L^2}
\le C_1(1+t)^{-\frac{3+2k}{4}}, k \in [0, N].
\end{gather*}
Here $t_*$ is a positive large time, $c_1$ and $C_1$ are two positive constants independent of time.
\end{rema}

\begin{rema}
After the completion of work in Theorem \ref{THM2}, the first author in \cite{Chen-Pan-Tong}
told us that they have also addressed the sharp time decay rates(including lower and upper bound decay rates)
for the isentropic Navier-Stokes system in three dimensional whole space under the $H^3-$framework.
However, these two works are done independently, and investigated the isentropic Navier-Stokes
and compressible Hall-MHD equations respectively.
\end{rema}

Next, we will establish the lower bound of decay rate for the time derivatives
of solution to the compressible Hall-MHD equation \eqref{eq-MHD}.

\begin{theo}\label{THM3}
Assume the condition \eqref{Lower-frequency-condition} holds on,
then the global classical solution $(\rho, u, B)$ obtained in Theorem \ref{THM1}
satisfies for all $t \ge t_*$
\begin{align}\label{Decay31}
c_1(1+t)^{-\frac{5}{4}}&\le \|\p_t u\|_{L^2} \le C_1(1+t)^{-\frac{5}{4}},\\
\label{Decay32}
c_1(1+t)^{-\frac{7}{4}}&\le \|\p_t B\|_{L^2} \le C_1(1+t)^{-\frac{7}{4}}.
\end{align}
Furthermore, if there exists a small constant $\d_1$ such that $\|u_0\|_{L^1}\le \d_1$,
it holds on for all $t \ge t_*$
\begin{align}\label{Decay33}
c_1(1+t)^{-\frac{5}{4}}&\le \|\p_t \rho\|_{L^2}   \le C_1(1+t)^{-\frac{5}{4}},\\
\label{Decay34}
c_1(1+t)^{-\frac{5}{4}}&\le \|{\rm div} u\|_{L^2} \le C_1(1+t)^{-\frac{5}{4}}.
\end{align}
Here $t_*$ is a positive large time, $c_1$ and $C_1$ are two positive constants independent of time.
\end{theo}

\begin{rema}
The lower bounds of decay rates for the time derivatives of density, velocity and magnetic field
for the compressible Navier-Stokes and Hall-MHD equations in $L^2-$norm are obtained for the first time.
\end{rema}

Now we turn to the weighted case.
For the incompressible Navier-Stokes equation, the weighted decay rate is
also widely studied by many mathematicians for strong solution in whole space.
\textbf{The basic question is the following:} Assuming
\begin{equation}\label{assumption01}
\|u(\cdot, t)\|_{L^2}=\mathcal{O}(t^{-\theta}),
\end{equation}
what is the rate of decay of $\|\x^{\g} u\|_{L^2}$, or more generally, the rate of
decay of $\|\x^{\g} u\|_{L^p}$.
In other words, we are interested in to what extent the temporal decay of $L^2$ norm
influence the decay rate of the weighted norm of velocity.
The main obstacle in obtaining sharp rate is the presence of the pressure term in the equation.
The study of space-time decay rate close to the heat equation was initiated by Takahashi \cite{Takahashi}
for the Navier-Stokes equation with nonzero forcing and zero initial data.
Then, Amrouche et al.\cite{Amrouche-Girault-Schonbek-Schonbek} studied
the pointwise behavior of solutions themselves and their derivatives with nonzero initial data but zero external force,
refined the method developed in \cite{Schonbek-Schonbek}, and deduced some space-time decay estimate.
However, their decay results are different from those of the heat equation.
Based on explicit solutions for the heat equation(refer to \cite{Amrouche-Girault-Schonbek-Schonbek}),
the conjectured optimal rate of decay for the incompressible Navier-Stokes equation is
\begin{equation}\label{assumption02}
\|\x^{\g} u(\cdot, t)\|_{L^2}=\mathcal{O}(t^{-\theta+{\g}/2}).
\end{equation}
Miyakawa \cite{Miyakawa} established the sharp space-time decay rate for solutions to the Navier-Stokes equation.
Based on a parabolic interpolation inequality, Kukavica \cite{Kukavica-01} also obtained the sharp decay rate
for any weighted norm of higher order, i.e., the quantity $\|\x^{\g} D^k_x u\|_{L^2}$ if the decay rates
\eqref{assumption01} and \eqref{assumption02} hold on.
This was further improved by Kukavica and Torres \cite{Kukavica-Torres-06} in sense of extending the weighted exponent.
Later, the assumption decay rate \eqref{assumption02}  was also verified in
their subsequent papers \cite{{Kukavica-Torres-07},{Kukavica-09}}.
Based on the series of papers \cite{{Kukavica-01},{Kukavica-Torres-06},{Kukavica-Torres-07},{Kukavica-09}},
Weng \cite{Weng-JFA} also obtained the sharp space-time decay for the incompressible viscous
resistive MHD and Hall-MHD equations.
\textit{To the authors' knowledge, there are no references addressing the space-time decay for the compressible
Navier-Stokes, viscous resistive MHD and  Hall-MHD equations.}


The last main result in this article is devoted to weighted decay rates of solution of \eqref{eq-MHD},
which is inspired by the work of Kukavica and Torres
\cite{{Kukavica-01},{Kukavica-Torres-06},{Kukavica-Torres-07},{Kukavica-09}},
and Weng \cite{Weng-JFA}, where the weighted decay estimates are established for the incompressible flow.

\begin{theo}\label{THM4}
Let $(\rho, u, B)$ be the strong solution to equations \eqref{eq-MHD}-\eqref{BC1}
with initial data $(\rho_0-1, u_0, B_0)$ belonging to the Schartz class.
Under the assumptions in Theorem \ref{THM1}, then it holds on
\begin{align}\label{Decay41}
\|\n^k (\rho-1)(t)\|_{H^{3-k}_\g}+\|\n^k u(t)\|_{H^{3-k}_\g}
&\le Ct^{-\frac{3}{4}+\g-\frac{k}{2}},\quad k=0,1,2;\\
\label{Decay42}
\|\n^k \p_t \rho(t)\|_{L^2_\g}+\|\n^k \p_t u(t)\|_{L^2_\g}
&\le C t^{-\frac{5}{4}+\g-\frac{k}{2}},\quad k=0,1;\\
\label{Decay43}
 \|\n^k B(t)\|_{L^2_\g}
&\le Ct^{-\frac{3}{4}+\frac{\g}{2}-\frac{k}{2}}, \quad k=0,1,2,3;\\
\label{Decay44}
\|\n^k \p_t B(t)\|_{L^2_\g}
&\le C t^{-\frac{7}{4}+\frac{\g}{2}-\frac{k}{2}},\quad k=0,1;
\end{align}
for all $\g \ge 0$.
Furthermore, if the Fourier transform $\widehat{B_0}=\mathcal{F}(B_0)$
satisfies $|\widehat{B_0}|\ge c_0$ for all $0\le |\xi| \ll 1$
with $c_0$ a positive constant, then the magnetic field given by Theorem \ref{THM1}
satisfies for $t \ge t_0$ with $t_0>0$ a sufficient large time that
\begin{gather}\label{Decay45}
c_1 t^{-\frac{7}{4}+\frac{\g}{2}} \le \|\p_t B(t)\|_{L^2_\g}
\le C_1 t^{-\frac{7}{4}+\frac{\g}{2}},\\
\label{Decay46}
c_1 t^{-\frac{3}{4}+\frac{\g}{2}-\frac{k}{2}}
\le \|\n^k B(t)\|_{L^2_\g} \le C_1 t^{-\frac{3}{4}+\frac{\g}{2}-\frac{k}{2}},
\end{gather}
for all $\g \in [0, 1]$, and $k=0,1,2$.
Here $c_1$ and $C_1$ are positive constants independent of time.
\end{theo}

\begin{rema}
For the incompressible flows(see \cite{{Kukavica-01},{Kukavica-Torres-06},{Kukavica-Torres-07},{Kukavica-09},{Weng-JFA}}),
the upper bound of decay rate for the $k(\ge0)-$th order
spatial derivative of velocity converges to zero in weighted space $L^2_\g(\g \in [0, 5/2))$
is $t^{-\frac{3}{4}+\frac{\g}{2}-\frac{k}{2}}$.
This decay rate is better than the rate for the compressible
flows obtained in \eqref{Decay41}.
However, the weighted index $\g$ in \eqref{Decay41} is required to be greater than zero
rather than $[0, 5/2)$ for the incompressible Navier-Stokes and Hall-MHD equations.
\end{rema}

\begin{rema}
By the Gagliardo-Nirenberg and interpolation inequalities, we can also
obtain the decay rate estimate in weighted space $L^p_\g(2\le p \le \infty)$
for the solutions of the compressible Hall-MHD equation \eqref{eq-MHD}.
\end{rema}

\begin{rema}
It is  shown that the lower and upper bound of decay rates for the time derivative
and $k(\in [0, 2])-$th order spatial derivatives of magnetic field converging to zero in weighted space
$L^2_\g(\g \in [0, 1])$ are $t^{-\frac{7}{4}+\frac{\g}{2}}$
and $t^{-\frac{3}{4}+\frac{\g}{2}-\frac{k}{2}}$ respectively.
As far as we know, the lower bounds of decay rate for the magnetic field in \eqref{Decay45}
and \eqref{Decay46} are given for the first time.
\end{rema}

Now we comment on the analysis in this paper.
Firs of all, we address the lower bound of decay rate for the higher order spatial derivative of solution
to the compressible Hall-MHD equation \eqref{eq-MHD}.
Let $(U, U_l)$ be the solution of nonlinear and linearized problem respectively.
Define the difference $U_\d:= U-U_l$, then we have for any integer $k$:
$
\|\n^k U\|_{L^2} \ge  \|\n^k U_l\|_{L^2}-\|\n^k U_\d\|_{L^2}.
$
If the solutions $U_l$ and $U_\d$ obey the assumptions
\begin{equation}\label{decay-a}
\|\n^k U_l\|_{L^2} \ge C_{l,k} (1+t)^{-\alpha_{l,k}},
\quad \|\n^k U_\d\|_{L^2} \le C_{\d,k} (1+t)^{-\alpha_{\d, k}},
\end{equation}
where $\alpha_{l,k} \le \alpha_{\d, k}+\var$.
If $\var>0$, then we have for large time $t$
$$
\|\n^k U\|_{L^2} \ge
C_{l,k} (1+t)^{-\alpha_{l,k}}-\frac{C_{\d, k} }{(1+t)^{\var}}(1+t)^{-(\alpha_{\d,k}+\var)}
\ge \frac{C_{l,k}}{2} (1+t)^{-\alpha_{l,k}}.
$$
If $\var=0$ and $C_{\d,k}$ is a small constant, then we have
$
\|\n^k U\|_{L^2} \ge \frac{1}{2}{C_{l,k}}(1+t)^{-\alpha_{l,k}}.
$
Since the lower bound of decay rate \eqref{decay-a} for the linearized part
can be obtained easily just by addressing the spectral analysis to
the differential operator of linearized part, see Proposition \ref{Decay-Linear} in section \ref{lower-bound}.
Then, the essential setup of lower bound decay rate for original nonlinear problem
is to obtain the upper bound of decay rate \eqref{decay-a} for the difference $\n^k U_\d$.
The upper bound of decay rate for $U_\d$ can be obtained easily just using the Duhamel principle formula
and upper bound decay estimate \eqref{up-decay}.
The upper bound of decay rate for $\n^k U_\d(k \ge 1)$ can be obtained by using the Fourier Splitting method
by Schonbek \cite{Schonbek} rather than the Gevrey estimates (see \cite{{Oliver-Titi},{Foias-Teman},{Weng-JDE}}).
This method has been applied to obtain the decay rate for higher order spatial derivative
of solution to the incompressible Navier-Stokes equation and compressible nematic liquid crystal
flows in whole space, see \cite{{Schonbek-Wiegner},{Schonbek-CPDE-1995},{Gao-Tao-Yao-JDE}}.

Next, the lower bound of decay rate for the time derivative of velocity and magnetic field can be obtained
by using the lower bound of first order spatial derivative and equation \eqref{eq-MHD}.
If we use the transport equation to obtain the lower bound of decay rate for the time derivative of density,
we need to get the lower bound for the divergence of velocity(i.e., ${\rm div} u$).
To achieve this target, we need to assume the smallness for the initial velocity in $L^1$.

Finally, we address the upper decay rate for the solution of problem \eqref{eq-MHD} in weighted space.
The upper decay rate of density, velocity and magnetic field in weighted space can be obtained by using the lemma
\ref{space-time-lemma} respectively.
To obtain the space-time decay rate for higher order spatial derivative, one method is to use the parabolic interpolation
inequality, see \cite{{Kukavica-01},{Weng-JFA}}.
However, the equations \eqref{eq-MHD} is a hyperbolic-parabolic coupling one.
For the integer $\g$, we use the Fourier Splitting method(see \cite{Schonbek}) and induction argument
with respect to weighted index $\g$ to obtain the upper decay rate for higher order spatial derivative in weighed space.
As for the real number $\g$, the upper bound of decay rate can be obtained just by using the interpolation inequality.
The lower bound of space-time decay rate is also addressed for magnetic field itself and higher order derivative.

The rest of this paper is organized as follows.
In section \ref{lower-bound}, we establish the lower bound of decay rate for the solution
itself and derivative.
In section \ref{space-time}, we address the upper bound of decay rate for the solution
itself and spatial derivative in weighted space.
Some technical estimates used in sections \ref{lower-bound} and \ref{space-time}
will be proved in section \ref{technical}.

\section{Lower Bound of Decay Rate}\label{lower-bound}

In this section, we will address the lower bound of decay rate for the solution itself and derivative.
To this end, the upper bound of decay rate for the difference between the nonlinear and linearized parts
will be established. Finally, we study the lower bound of decay rate for the solution of
higher order spatial derivative and time derivative.

\subsection{Lower Bound of Decay Rate for Spatial Derivative}

In this subsection, we will establish lower bound of decay rate for
the higher order spatial derivative of classical solution to the compressible Hall-MHD equation \eqref{eq-MHD}.
For the sake of simplicity, we assume $P'(1)=1$ as mentioned before.
Let us denote $\vr := \rho-1$ and $m := \rho u$,
we rewrite \eqref{eq-MHD} in the perturbation form as
\begin{equation}\label{eq-per}
\left\{
\begin{aligned}
&\vr_t+{\rm div}m=0,\\
&m_t-\mu \Delta m-(\mu+\nu)\n {\rm div} m+\nabla \varrho=-{\rm div} S_1,\\
&B_t-\Delta B=\nt S_2, \quad {\rm div}B=0,
\end{aligned}
\right.
\end{equation}
where the function $S_1=S_1(\vr, u, B)$ and $S_2=S_2(\vr, u, B)$ are defined as
\begin{equation}\label{force1}
\left\{
\begin{aligned}
S_1=
&(1+\vr)u\otimes u+\mu \n (\vr u)+(\mu+\nu){\rm div}(\vr u)\mathbb{I}_{3\times 3}\\
&+(P(1+\vr)-P(1)-\vr)\mathbb{I}_{3\times 3}+\frac{1}{2}|B|^2\mathbb{I}_{3\times 3}-B \otimes B;\\
S_2=&u \times B-\frac{(\nt B)\times B}{1+\vr}.
\end{aligned}
\right.
\end{equation}
The initial data is given as
\begin{equation}\label{initial1}
\left.(\varrho, m, B)(x,t)\right|_{t=0}=(\varrho_0, m_0, B_0)(x)
\rightarrow(0,0,0) \quad {\text {as} } \quad |x|\rightarrow \infty.
\end{equation}
In order to obtain the lower decay estimate, we need to analysis the linearized part:
\begin{equation}\label{eq-linearized}
\left\{
\begin{aligned}
&\p_t \vr_l+{\rm div} m_l=0,\\
&\p_t m_l-\mu \Delta m_l-(\mu+\nu)\n {\rm div} m_l+\n \vr_l=0,\\
&\p_t B_l-\Delta B_l=0, \quad {\rm div} B_l=0.
\end{aligned}
\right.
\end{equation}
with the initial data
\begin{equation}\label{initial2}
\left.(\vr_l, m_l, B_l)(x,t)\right|_{t=0}=(\vr_0, m_0, B_0)(x)
\rightarrow(0,0,0) \quad {\text {as} } \quad |x|\rightarrow \infty.
\end{equation}
Here the initial data for the linearized part \eqref{eq-linearized} are the same as
the nonlinear part \eqref{eq-per}.
The following properties on the decay in time, which
can be found in \cite{{Hu-Wu},{Li-Zhang}}.
\begin{prop}\label{Decay-Linear}
Let $k\ge 0$ be an integer, and assume the Fourier transform
$\mathcal{F}(\vr_0, m_0, B_0):=(\widehat{\vr_0}, \widehat{m_0},\widehat{B_0})$
satisfies
$|(\widehat{\vr_0}, \widehat{m_0}, \widehat{B_0})|\le C|\xi|^\eta$ for $0\le |\xi| \ll 1$.
Then, the solution $(\vr_l, m_l, B_l)$ of linearized system \eqref{eq-linearized}
has the following estimates for all $t \ge 0$
\begin{equation}\label{up-decay}
\|\nabla^k(\vr_l, m_l, B_l)(t)\|_{L^2}
\le C(1+t)^{-(\frac{3}{4}+\frac{\eta}{2}+\frac{k}{2})}
    (\|(\widehat{\vr_0}, \widehat{m_0}, \widehat{B_0})\|_{L^\infty}+\|\nabla^k(\vr_0, m_0, B_0)\|_{L^2}).
\end{equation}
If the Fourier transform
$\mathcal{F}(\vr_0, m_0, B_0)=(\widehat{\vr_0}, \widehat{m_0},\widehat{B_0})$
satisfies
$$
|\widehat{\vr_0}|\ge c_0, ~|\widehat{B_0}|\ge c_0,~ \widehat{m_0}=0, ~0 \le |\xi| \ll 1,
$$
with $c_0>0$ a constant, then we have for large time $t$
\begin{equation}\label{lower-decay}
\min \{\|\n^k \vr_l(t)\|_{L^2}, \|\n^k m_l(t)\|_{L^2}, \|\n^k B_l(t)\|_{L^2\}}
\ge c_1 (1+t)^{-(\frac{3}{4}+\frac{k}{2})},
\end{equation}
where $c_1$ and $C$ are positive constants independent of time $t$.
\end{prop}

Indeed, the lower bound of decay rate \eqref{lower-decay} is only established
for the case $k=0$, whereas the general case $k(\ge 1)$ can be obtained just following
the method in \cite{Li-Zhang}.
In order to obtain the lower bound for the solution of the compressible Hall-MHD equation \eqref{eq-per},
we need to address the upper decay rate for the difference between the nonlinear and linearized part.
Hence, let us denote
$$
\vr_\d :=\vr-\vr_l, m_\d :=m-m_l, B_\d :=B-B_l,
$$
then they satisfy the following system
\begin{equation}\label{eq-differ}
\left\{
\begin{aligned}
&\p_t \vr_\d+{\rm div} m_\d=0,\\
&\p_t m_\d-\mu \Delta m_\d-(\mu+\nu)\nabla {\rm div} m_\d+\nabla \varrho_\d=-{\rm div} S_1,\\
&\p_t B_\d-\Delta B_\d=\nt S_2, \quad {\rm div}B_\d=0,
\end{aligned}
\right.
\end{equation}
with the zero initial data
\begin{equation}\label{initial3}
\left.(\vr_\d, m_\d, B_\d)(x,t)\right|_{t=0}=(0, 0, 0).
\end{equation}
Now we will establish the decay rate for the solution $(\vr_\d, m_\d, B_\d)$
of equation \eqref{eq-differ} in the following.

\begin{lemm}\label{energy1}
For any smooth solution $(\vr_\d, m_\d)$ of the equation \eqref{eq-differ}, it holds on
\begin{equation}\label{claim1}
\frac{d}{dt}\|\n^k (m_\d, \vr_\d)\|_{H^{2-k}}^2+C\|\n^{k+1} m_\d\|_{H^{2-k}}^2
\le C\|\n^k(\vr, u, B)\|_{H^{2-k}}^2 \|\n^k (\vr, u, B)\|_{H^{3-k}}^2,
\end{equation}
and
\begin{equation}\label{claim2}
\begin{aligned}
&\sum_{l=k}^1\frac{d}{dt}\int \n^l m_\d \cdot \n^{l+1} \vr_\d dx
+\frac{1}{2}\|\n^{k+1} \vr_\d\|_{H^{1-k}}^2\\
&\le C\|\n^{k+1} m_\d\|_{H^{2-k}}^2
    +C\|\n (\vr, u, B)\|_{H^{1}}^2 \|\n^{k+1} (\vr, u, B)\|_{H^{2-k}}^2,
\end{aligned}
\end{equation}
where $k=0,1$.
\end{lemm}

It should be pointed out that we only establish the $H^2$ energy estimates
in Lemma \ref{energy1} although the initial data belong to $H^3$.
The reason is the appearance of second order spatial derivatives of density
in the nonlinear term $\nc S_1$.
The above inequalities \eqref{claim1} and \eqref{claim2} in Lemma \ref{energy1}
will be proved later in Section \ref{technical}.
Multiplying inequality \eqref{claim2} by a small constant $\d$,
adding with \eqref{claim1} and using the decay rate \eqref{Decay1},
we get for all $t \ge 0$
\begin{equation}\label{eg01}
\frac{d}{dt}\mathcal{E}_l^2(t)
+\frac{\d}{2}\|\n^{l+1}\vr_\d\|_{H^{1-l}}^2
+\frac{C_1}{2}\|\n^{l+1}m_\d \|_{H^{2-l}}^2
\le C(1+t)^{-(3+2l)},
\end{equation}
where $l=0,1$.
Here the energy $\mathcal{E}_l^2(t)$ is defined by
\begin{equation}\label{e2}
\mathcal{E}_l^2(t):=
\sum_{k=l}^2 \|\n^k (\vr_\d, m_\d)\|_{L^2}^2
+\d \sum_{k=l}^1 \int \n^k m_\d \cdot \n^{k+1}\vr_\d dx.
\end{equation}
Due to the smallness of $\d$, there are two constants $C_*$ and $C^*$(independent of time) such that
\begin{equation}\label{equivalent}
C_* \|\n^l (\vr_\d, m_\d)(t)\|_{H^{2-l}}^2
\le \mathcal{E}_l^2(t)
\le C^* \|\n^l (\vr_\d, m_\d)\|_{H^{2-l}}^2.
\end{equation}

Now we establish the upper bound decay rate of solution $(\vr_\d, m_\d)$ for the
equation \eqref{eq-differ}.

\begin{lemm}\label{upper1}
Under the assumptions in Theorem \ref{THM1}, then the smooth solution $(\vr_\d, m_\d, B_\d)$
of equation \eqref{eq-differ} satisfies
\begin{equation}\label{231}
\|\n^k (\vr_\d, m_\d)(t)\|_{H^{2-k}}\le C(1+t)^{-\frac{5+2k}{4}},k=0,1;
\end{equation}
for all large time $t$.
\end{lemm}

\begin{proof}
Taking $l=0$ in \eqref{eg01}, then we have
\begin{equation*}
\frac{d}{dt}\mathcal{E}_0^2(t)
+C(\|\n \vr_\d\|_{H^{1}}^2+\|\n m_\d \|_{H^{2}}^2)\le C(1+t)^{-3}.
\end{equation*}
Obviously, the dissipation term $\|\n \vr_\d\|_{H^{1}}^2+\|\n m_\d \|_{H^{2}}^2$
can not control the energy term $\mathcal{E}_0^2(t)$ in above inequality.
Thus, we add both sides of the above inequality
with term $\|(\vr_\d, m_\d)\|_{L^2}^2$, and hence get
\begin{equation}\label{232}
\frac{d}{dt}\mathcal{E}_0^2(t)
+C(\|\vr_\d\|_{H^{2}}^2+\|m_\d \|_{H^{3}}^2)
\le \|(\vr_\d, m_\d)\|_{L^2}^2+C(1+t)^{-3}.
\end{equation}
By virtue of the Duhamel principle formula and estimate \eqref{up-decay}, we have
\begin{equation}\label{233}
\begin{aligned}
\|(\vr_\d, m_\d)(t)\|_{L^2}
&\le \int_0^t (1+t-\tau)^{-\frac{5}{4}}(\||\xi|^{-1}\mathcal{F}(\nc S_1)\|_{L^\infty}+\|\n S_1\|_{L^2})d\tau\\
&\le C\int_0^t (1+t-\tau)^{-\frac{5}{4}}(\|S_1\|_{L^1}+\|\n S_1 \|_{L^2})d\tau\\
&\le C\int_0^t (1+t-\tau)^{-\frac{5}{4}}(1+\tau)^{-\frac{3}{2}}d\tau\\
&\le C(1+t)^{-\frac{5}{4}},
\end{aligned}
\end{equation}
where we have used the basic fact
$$
\|S_1\|_{L^1}+\|\n S_1 \|_{L^2}
\le C\|(\vr, u, B)(t)\|_{H^2}^2
\le C(1+t)^{-\frac{3}{2}}.
$$
Using \eqref{232}, \eqref{233} and equivalent relation \eqref{equivalent}, one arrives at
\begin{equation*}
\frac{d}{dt}\mathcal{E}_0^2(t)+CC_* \mathcal{E}_0^2(t) \le C(1+t)^{-\frac{5}{2}},
\end{equation*}
which implies directly
\begin{equation*}
\mathcal{E}_0^2(t)\le C\int_0^t e^{-C(t-\tau)}(1+\tau)^{-\frac{5}{2}}d\tau \le C(1+t)^{-\frac{5}{2}}.
\end{equation*}
Hence, the combination of the above estimate and equivalent relation \eqref{equivalent} yields
\begin{equation}\label{234}
\|(\vr_\d, m_\d)(t)\|_{H^{2}} \le C(1+t)^{-\frac{5}{4}}
\end{equation}
for all $t \ge 0$.

Next, taking $l=1$ in inequality \eqref{eg01}, we have
\begin{equation}
\frac{d}{dt}\mathcal{E}_1^2(t)+C(\|\n^2\vr_\d\|_{L^2}^2+\|\n^2 m_\d \|_{H^{1}}^2) \le C(1+t)^{-5}.
\end{equation}
In order to obtain the time decay rate for the first order spatial derivative
of solution by the Fourier Splitting Method(by Schonbek \cite{Schonbek}),
which has been applied to obtain decay rate for the incompressible Navier-Stokes
equations in higher order derivative norm(see \cite{Schonbek-Wiegner}).
The difficulty, arising from the compressible Navier-Stokes equations, is the appearance of density
that obeys the transport equation rather diffusive one.
To get rid of this difficulty, our idea is write the above inequality  as the following form
\begin{equation}\label{235}
\frac{d}{dt}\mathcal{E}_1^2(t)
+\frac{C}{2}\|\n^2\vr_\d\|_{L^2}^2
+\frac{C}{2}\|\n^2\vr_\d\|_{L^2}^2
+\frac{C}{2}\|\n^2 m_\d \|_{H^{1}}^2 \le C(1+t)^{-5}.
\end{equation}
For some constant $R$ defined below, denoting the time sphere(see \cite{Schonbek})
$$
S_0:=\left\{\left. \xi\in \mathbb{R}^3\right| |\xi|\le \left(\frac{R}{1+t}\right)^\frac{1}{2}\right\},
$$
it follows immediately
$$
\begin{aligned}
\int_{\mathbb{R}^3} |\nabla^{2} \vr_\d|^2 dx
&\ge \int_{\mathbb{R}^3/S_0} |\xi|^4 |\widehat{\vr_\d}|^2d\xi\\
&\ge \frac{R}{1+t}\int_{\mathbb{R}^3/S_0} |\xi|^2 |\widehat{\vr_\d}|^2d\xi\\
&\ge \frac{R}{1+t}\int_{\mathbb{R}^3} |\xi|^2 |\widehat{\vr_\d}|^2 d\xi
     -\left(\frac{R}{1+t}\right)^2\int_{S_0} |\widehat{\vr_\d}|^2 d\xi,
\end{aligned}
$$
or equivalently
\begin{equation}\label{Fourier}
\|\nabla^2 \vr_\d\|_{L^2}^2
\ge \frac{R}{1+t}\|\nabla \vr_\d\|_{L^2}^2-\left(\frac{R}{1+t}\right)^2\| \vr_\d\|_{L^2}^2.
\end{equation}
Similarly, we also obtain
\begin{equation*}
\|\n^2 m_\d \|_{H^{1}}^2
\ge \frac{R}{1+t} \|\n m_\d \|_{H^{1}}^2
    -\frac{R^2}{(1+t)^2} \|m_\d \|_{H^{1}}^2.
\end{equation*}
Thus, one arrives at
\begin{equation}\label{236}
\begin{aligned}
&\frac{d}{dt}\mathcal{E}_1^2(t)
+\frac{CR}{2(1+t)}\|\n \vr_\d\|_{L^2}^2
+\frac{C}{2}\|\n^2 \vr_\d\|_{L^2}^2+\frac{CR}{2(1+t)}\|\n m_\d \|_{H^{1}}^2 \\
&\le \frac{C R^2}{(1+t)^2} (\|\vr_\d \|_{L^2}^2+\|m_\d \|_{H^{1}}^2)+C(1+t)^{-5}\\
&\le C R^2(1+t)^{-\frac{9}{2}}+C(1+t)^{-5},
\end{aligned}
\end{equation}
where we have used the decay rate \eqref{234}.
Choose $t \ge T_{1}:=R-1$ such that
$$
\frac{R}{1+t}\le 1,
$$
and hence we have
\begin{equation*}
\frac{d}{dt}\mathcal{E}_1^2(t)
+\frac{CR}{2(1+t)}\|\n (\vr_\d, m_\d)\|_{H^{1}}^2
\le C  R^2(1+t)^{-\frac{9}{2}}+C(1+t)^{-5}.
\end{equation*}
Obliviously, the energy $\mathcal{E}_1^2(t)$ is equivalent to the norm $\|\n (\vr_\d, m_\d)\|_{H^{1}}^2$.
And hence, the advantage of the form \eqref{235} is that the dissipation in \eqref{234}
can control the energy  after using the Fourier Splitting Method.
Using the equivalent relation \eqref{equivalent}, we can get
\begin{equation*}
\frac{d}{dt}\mathcal{E}_1^2(t)+\frac{CR}{2C^{*}(1+t)}\mathcal{E}_1^2(t)\le C R^2(1+t)^{-\frac{9}{2}}+C(1+t)^{-5}.
\end{equation*}
Choosing $R=8C^{*}/C$ and multiplying the above inequality by $(1+t)^4$, one arrives at
\begin{equation*}
\frac{d}{dt}[(1+t)^4 \mathcal{E}_1^2(t)]\le C(1+t)^{-\frac{1}{2}}.
\end{equation*}
Then, the integration over $[T_1, t]$ gives directly
\begin{equation*}
\|\n (\vr_\d, m_\d)(t)\|_{H^{1}}^2 \le C(1+t)^{-\frac{7}{2}},
\end{equation*}
where we have used the equivalent relation \eqref{equivalent}
and uniform estimate \eqref{uniform}.
Therefore we complete the proof of lemma.
\end{proof}

Next, we claim the following energy estimates for the magnetic field,
which will be proved in Section \ref{technical}.

\begin{lemm}
For any smooth solution $(\vr_\d, m_\d, B_\d)$ of the equation \eqref{eq-differ}, it holds on
\begin{equation}\label{claim3}
\frac{d}{dt} \int |\n^k B_\d|^2 dx+\int |\n^{k+1} B_\d|^2 dx
\le C \|(u, B)\|_{W^{1,\infty}}^2\|\n^k (\vr, u, B)\|_{L^2}^2
    +C\|B\|_{L^\infty}^2\|\n^{k+1} B\|_{L^2}^2,
\end{equation}
for $k=1,2,3$.
\end{lemm}

Now, we will establish the upper decay rate for the difference of magnetic field
between the nonlinear and linearized parts.

\begin{lemm}\label{upper1}
Under the assumptions in Theorem \ref{THM1}, then the smooth solution $(\vr_\d, m_\d, B_\d)$
of equation \eqref{eq-differ} satisfies
\begin{equation}\label{251}
\|\n^k B_\d(t)\|_{L^2}\le C(1+t)^{-\frac{5+2k}{4}}, k=0,1,2,3,
\end{equation}
for all large time $t$.
\end{lemm}

\begin{proof}
By virtue of the Duhamel principle formula, estimate \eqref{up-decay} and decay rate \eqref{Decay1}, we have
\begin{equation}\label{252}
\begin{aligned}
\|B_\d(t)\|_{L^2}
&\le \int_0^t (1+t-\tau)^{-\frac{5}{4}}(\||\xi|^{-1}\mathcal{F}(\nt S_2)\|_{L^\infty}+\|\n S_2\|_{L^2})d\tau\\
&\le C\int_0^t (1+t-\tau)^{-\frac{5}{4}}(\| S_2\|_{L^1}+\|\n  S_2\|_{L^2})d\tau\\
&\le C\int_0^t (1+t-\tau)^{-\frac{5}{4}}\|(\vr, u, B)\|_{H^3}^2d\tau\\
&\le C\int_0^t (1+t-\tau)^{-\frac{5}{4}}(1+\tau)^{-\frac{3}{2}}d\tau\\
&\le C(1+t)^{-\frac{5}{4}}.
\end{aligned}
\end{equation}

Taking $k=1$ in inequality \eqref{claim3} and using decay rate \eqref{Decay1}, we have
\begin{equation*}
\frac{d}{dt}\|\n B_\d\|_{L^2}^2+\|\n^{2} B_\d\|_{L^2}^2 \le C(1+t)^{-5}.
\end{equation*}
Similar to inequality \eqref{Fourier}, we use decay rate \eqref{252} to get
\begin{equation*}
\frac{d}{dt}\|\n B_\d\|_{L^2}^2+\frac{R}{1+t}\|\n B_\d\|_{L^2}^2
\le \frac{R^2}{(1+t)^2}\|B_\d\|_{L^2}^2+C(1+t)^{-5}
\le CR^2(1+t)^{-\frac{9}{2}}.
\end{equation*}
Choosing $R=4$, multiplying by $(1+t)^4$ and integrating with time, we have
\begin{equation*}
\|\n B_\d(t)\|_{L^2}^2 \le C(1+t)^{-\frac{7}{2}}.
\end{equation*}
Similarly, we can obtain
\begin{equation*}\label{253}
\|\n^2 B_\d(t)\|_{L^2}^2 \le C(1+t)^{-\frac{9}{2}}.
\end{equation*}

Finally, we establish the decay rate for $\|\n^3 B_\d (t)\|_{L^2}$.
The difficulty comes from the term $\|\n^4 B(t)\|_{L^2}$ on the righthand side of inequality \eqref{claim3}.
The idea is to used the time weighted integration for the dissipation term $\|\n^4 B(t)\|_{L^2}$.
Similar to inequality \eqref{Fourier}, we can obtain
\begin{equation*}
\|\nabla^4 B_\d\|_{L^2}^2
\ge \frac{R}{1+t}\|\nabla^3 B_\d\|_{L^2}^2
-\left(\frac{R}{1+t}\right)^2\|\n ^2 B_\d\|_{L^2}^2,
\end{equation*}
and hence, one arrives at
\begin{equation*}
\begin{aligned}
&\frac{d}{dt}\|\n^3 B_\d\|_{L^2}^2 +\frac{R}{1+t}\|\n^{3} B_\d\|_{L^2}^2\\
&\le  \frac{R^2}{(1+t)^2}\|\n^{2} B_\d\|_{L^2}^2
      +C \|(u, B)\|_{W^{1,\infty}}^2\|\n^3 (\vr, u, B)\|_{L^2}^2+C\|B\|_{L^\infty}^2\|\n^4 B\|_{L^2}^2\\
&\le \frac{R^2}{(1+t)^2}\|\n^{2} B_\d\|_{L^2}^2
      +C(1+t)^{-\frac{13}{2}}+C(1+t)^{-3}\|\n^4 B\|_{L^2}^2\\
&\le C(1+t)^{-\frac{13}{2}}+C(1+t)^{-3}\|\n^4 B\|_{L^2}^2.
\end{aligned}
\end{equation*}
Choose $R=6$ and multiply the above inequality by $(1+t)^6$, it holds on
\begin{equation*}
\frac{d}{dt}\{(1+t)^6\|\n^3 B_\d\|_{L^2}^2\}
\le C(1+t)^{-\frac{1}{2}}+(1+t)^3 \|\n^4 B\|_{L^2}^2,
\end{equation*}
which, integrating over $[T_*, t]$, yields directly
\begin{equation*}\label{254}
\begin{aligned}
(1+t)^6\|\n^3 B_\d\|_{L^2}^2
\le (1+T_*)^6\|\n^3 B_\d(T_*)\|_{L^2}^2
     +C((1+t)^{\frac{1}{2}}-(1+T_*)^{\frac{1}{2}})
     +\int_{T_*}^t(1+\tau)^3 \|\n^4 B\|_{L^2}^2 d\tau.
\end{aligned}
\end{equation*}
We claim the estimate(it will be proved in section \ref{technical})
\begin{equation}\label{Claim-Decay}
\int_{T_*}^t (1+\tau)^3\|\n^4 B\|_{L^2}^2 d\tau \le C,
\end{equation}
where $C$ is positive constant independent of time.
Then, we can obtain
\begin{equation*}
\|\n^3 B_\d(t)\|_{L^2}^2\le C(1+t)^{-\frac{11}{2}},
\end{equation*}
where we have used the uniform estimate \eqref{uniform}. Therefore we complete the proof of lemma.
\end{proof}

Finally, we establish the lower bound of time decay rate
for the global solution of compressible Hall-MHD equation \eqref{eq-MHD}.

\begin{lemm}\label{Lower}
Under all the assumptions of Theorem \ref{THM2},  then the global solution $(\rho, u, B)$
of compressible Hall-MHD equation \eqref{eq-MHD} has the following estimates for all $t \ge t_*$
\begin{equation}\label{261}
\begin{aligned}
&\min \{\|\n^k (\rho-1) (t)\|_{L^2}, \|\n^k u(t)\|_{L^2}\}\ge C(1+t)^{-\frac{3+2k}{4}}, k=0,1;\\
&\|\n^k B(t)\|_{L^2}\ge C(1+t)^{-\frac{3+2k}{4}}, k=0,1,2,3,
\end{aligned}
\end{equation}
Here $t_*$ is a positive large time, and $C$ is a constant independent of time.
\end{lemm}

\begin{proof}
Remember the definition
$$
\vr_\d=\vr-\vr_l, m_\d=m-m_l, B_\d=B-B_l,
$$
we have
\begin{equation*}
\|\vr_l\|_{L^2}=\|\vr-\vr_\d\|_{L^2}\le \|\vr\|_{L^2}+\|\vr_\d\|_{L^2},
\end{equation*}
which, together with estimates \eqref{lower-decay} and \eqref{231}, yields directly
\begin{equation*}
\begin{aligned}
\|\vr (t)\|_{L^2}
&\ge \|\vr_l\|_{L^2}-\|\vr_\d\|_{L^2}\\
&\ge C_2(1+t)^{-\frac{3}{4}}-C_3(1+t)^{-\frac{5}{4}}\\
&\ge C_2(1+t)^{-\frac{3}{4}}-\frac{C_3}{(1+t)^{\frac{1}{2}}}(1+t)^{-\frac{3}{4}}.
\end{aligned}
\end{equation*}
Choosing $ t \ge \frac{4C_3^2-C_2^2}{C_2^2}$, it holds on
\begin{equation*}
\|\vr (t)\|_{L^2} \ge C(1+t)^{-\frac{3}{4}}.
\end{equation*}
Similarly, using estimates \eqref{lower-decay}, \eqref{231} and \eqref{251}, we also have
\begin{equation*}
\begin{aligned}
&\|\n \vr(t)\|_{L^2}\ge C(1+t)^{-\frac{5}{4}},\\
&\|\n^{k} m (t)\|_{L^2}\ge C(1+t)^{-\frac{3+2k}{4}},\ k=0,1;\\
&\|\n^{k} B (t)\|_{L^2}\ge C(1+t)^{-\frac{3+2k}{4}},\ k=0,1,2,3.
\end{aligned}
\end{equation*}

Finally, we establish the lower decay rate for the velocity.
Using decay \eqref{Decay1}, we get
\begin{equation*}
\|m\|_{L^2} \le \|u\|_{L^2}+\|\vr\|_{L^\infty}\|u\|_{L^2} \le \|u\|_{L^2}+C(1+t)^{-\frac{9}{4}},
\end{equation*}
which, together with \eqref{lower-decay}, yields for large time $t$
\begin{equation*}
\begin{aligned}
\|u\|_{L^2}
\ge \|m\|_{L^2}-C(1+t)^{-\frac{9}{4}}
\ge C(1+t)^{-\frac{3}{4}}-C(1+t)^{-\frac{9}{4}}
\ge C(1+t)^{-\frac{3}{4}}.
\end{aligned}
\end{equation*}
Similar, we have that
\begin{equation*}
\begin{aligned}
\|\n m\|_{L^2}
&\le \|\n u\|_{L^2}+\|\n (\vr u)\|_{L^2}\\
&\le \|\n u\|_{L^2}+\|\vr\|_{L^\infty}\|\n u\|_{L^2}+\|u\|_{L^\infty}\|\n \vr\|_{L^2}\\
&\le \|\n u\|_{L^2}+C(1+t)^{-\frac{11}{4}},
\end{aligned}
\end{equation*}
and hence, one arrives at for large time $t$
\begin{equation*}
\begin{aligned}
\|\n u\|_{L^2}
\ge \|\n m\|_{L^2}-C(1+t)^{-\frac{9}{4}-\frac{1}{2}}
\ge C(1+t)^{-\frac{5}{4}}-C(1+t)^{-\frac{11}{4}}
\ge C(1+t)^{-\frac{5}{4}}.
\end{aligned}
\end{equation*}
Therefore, we complete the proof of lemma.
\end{proof}

\subsection{Lower Bound of Decay Rate for Time Derivative}

In this subsection, we will establish the lower bound of decay rate for the time derivative
of density, velocity and magnetic field.
For the sake of simplicity, we assume $P'(1)=1$ as mentioned before.
Then, denoting $\varrho :=\rho-1$, we rewrite \eqref{eq-MHD} in the perturbation form as
\begin{equation}\label{eq-per2}
\left\{
\begin{aligned}
&\p_t \varrho+{\rm div} u=G_1,\\
&\p_t u-\mu \Delta u-(\mu+\nu)\nabla {\rm div} u+\nabla \varrho=G_2,\\
&\p_t B-\Delta B=G_3, \quad {\rm div} B=0,
\end{aligned}
\right.
\end{equation}
where the function $G_i(i=1,2,3)$ is defined as
\begin{equation*}
\left\{
\begin{aligned}
&G_1=-\vr {\rm div} u-u\cdot \n \vr,\\
&G_2=-u \cdot \!\nabla u+(\frac{1}{\vr+1}-1)[\mu\Delta u+(\mu+\nu)\n {\rm div} u]
 -\left\{\frac{P'(\vr+1)}{\vr+1}-1\right\}\nabla \varrho+\frac{(\nt B)\times B }{1+\vr},\\
&G_3=\nt (u \times B)-\nt \left\{\frac{(\nt B)\times B }{1+\vr}\right\}.
\end{aligned}
\right.
\end{equation*}
The initial data are given by
\begin{equation}\label{initial4}
\left.(\vr, u, B)(x,t)\right|_{t=0}=(\vr_0, u_0, B_0)(x)
\rightarrow(0,0,0) \quad {\text {as} } \quad |x|\rightarrow \infty.
\end{equation}

Now, we establish the lower bound decay rate for the time derivative of solution in $L^2$ norm.

\begin{lemm}
Under the assumptions in Theorem \ref{THM2}, then the global solution $(\vr, u, B)$
of equation \eqref{eq-per2} has the following estimates
\begin{gather}\label{271}
\min \{\|\p_t \vr(t)\|_{L^2}, \|\p_t u(t)\|_{L^2}, \|{\rm div} u(t)\|_{L^2}\}\ge C(1+t)^{-\frac{5}{4}},\\
\label{272}
\|\p_t B(t)\|_{L^2}\ge C(1+t)^{-\frac{7}{4}},
\end{gather}
for all $t \ge t_*(t_*$ being a positive large time). Here $C$ is a positive constant independent of time.
\end{lemm}

\begin{proof}
First of all, we establish lower bound time decay rate for $\p_t B$ in $L^2-$norm.
Indeed, using the magnetic field equation in \eqref{eq-per2}, we have
\begin{equation*}
\|\p_t B\|_{L^2}
\ge \|\Delta B\|_{L^2}-\|G_3\|_{L^2}
\ge C(1+t)^{-\frac{7}{4}}-\|G_3\|_{L^2},
\end{equation*}
Using the Sobolev inequality and decay rate \eqref{Decay1}, we have
\begin{equation*}
\|G_3\|_{L^2}
\le C\|\n (u, B)\|_{H^1}\|\n (\vr, u, B)\|_{L^2}
    +C\|\n B\|_{H^1}\|\n^2 B\|_{L^2}
\le C(1+t)^{-\frac{5}{2}}.
\end{equation*}
And hence, it holds on
\begin{equation*}
\|\p_t B\|_{L^2}
\ge C(1+t)^{-\frac{7}{4}}-C(1+t)^{-\frac{5}{2}}
\ge  C(1+t)^{-\frac{7}{4}},
\end{equation*}
for all some large time $t$.

Next, we establish lower bound time decay rate for $\p_t u$ in $L^2-$norm.
Using the momentum equation in \eqref{eq-per2}, we have
\begin{equation*}
\|\n \vr \|_{L^2}\le \|\p_t u\|_{L^2} +\|\n^2 u\|_{L^2} +\|G_2\|_{L^2}.
\end{equation*}
And hence, we get
\begin{equation}\label{273}
\|\p_t u\|_{L^2}
\ge \|\n \vr \|_{L^2}-\|\n^2 u\|_{L^2}-\|G_2\|_{L^2}
\ge C(1+t)^{-\frac{5}{4}}-C(1+t)^{-\frac{7}{4}}-\|G_2\|_{L^2}.
\end{equation}
By virtue of the Sobolev inequality and time decay rate \eqref{Decay1}, we have
\begin{equation*}
\|G_2\|_{L^2} \le C\|\n (\vr, u, B)\|_{H^1}\|\n (u, B)\|_{H^1}\le C(1+t)^{-\frac{5}{2}}.
\end{equation*}
This and the inequality \eqref{273} yield directly
\begin{equation}\label{274}
\|\p_t u\|_{L^2}
\ge C(1+t)^{-\frac{5}{4}}-C(1+t)^{-\frac{7}{4}}-C(1+t)^{-\frac{5}{2}}
\ge C(1+t)^{-\frac{5}{4}},
\end{equation}
for some large time $t$.

Finally, we establish lower bound time decay rate for $\p_t \vr$ in $L^2-$norm.
To achieve this target, we use the transport equation in equation \eqref{eq-per2} to obtain
\begin{equation*}
\|{\rm div} u\|_{L^2}\le \|\p_t \vr\|_{L^2}+\|G_1\|_{L^2}.
\end{equation*}
By virtue of the Sobolev inequality and decay rate \eqref{Decay1}, it is easy to check that
\begin{equation*}
\|G_1\|_{L^2}\le C\|\n (\vr, u)\|_{H^1}^2 \le C(1+t)^{-\frac{5}{2}},
\end{equation*}
and hence, we obtain
\begin{equation}\label{275}
\|\p_t \vr\|_{L^2} \ge \|{\rm div} u\|_{L^2}-C(1+t)^{-\frac{5}{2}}.
\end{equation}
Now, we need to establish the lower bound decay rate for $\|{\rm div} u\|_{L^2}$.
Notice the relation differential relation $\Delta=\n {\rm div}-\nt \nt$, we get
$$
\|\n u\|_{L^2}^2=\|{\rm div} u\|_{L^2}^2+\|\nt u\|_{L^2}^2.
$$
And hence, one arrives at
\begin{equation}\label{276}
\|{\rm div} u\|_{L^2}
\ge C\|\n u\|_{L^2}-C\|\nt u\|_{L^2}
\ge C(1+t)^{-\frac{5}{4}}-C\|\nt u\|_{L^2},
\end{equation}
which implies that we need to establish upper bound decay rate for $\|\nt u\|_{L^2}$.
To this end, we take the $\nt$ operator the the velocity equation in \eqref{eq-per2} to get
$$
\p_t(\nt u)-\mu \Delta (\nt u)=\nt G_2.
$$
Using  Sobolev inequality, uniform bound \eqref{uniform} and decay rate \eqref{Decay1}, we have
\begin{equation*}
\|G_2\|_{L^1}+\|G_2\|_{L^2}
\le C(\|(\vr, u, B)\|_{L^2}+\|\n (u, B)\|_{H^1})\|\n (\vr, u, B)\|_{H^1}
\le C\d_0(1+t)^{-\frac{5}{4}}.
\end{equation*}
By virtue of the Duhamel principle formula and estimate \eqref{up-decay}, we get
\begin{equation*}
\begin{aligned}
\|\nt u\|_{L^2}\le
&C(1+t)^{-\frac{5}{4}}(\|\Lambda^{-1}\mathcal{F}(\nt u_0)\|_{L^\infty}
                          +\|\Lambda^{-1}\mathcal{F}(\nt u_0)\|_{L^2})\\
&+C\int_0^t (1+t-\tau)^{-\frac{5}{4}}
    (\|\Lambda^{-1}\mathcal{F}(\nt G_2)\|_{L^\infty}
                          +\|\Lambda^{-1}\mathcal{F}(\nt G_2)\|_{L^2})d\tau\\
\le &C(1+t)^{-\frac{5}{4}}(\|u_0\|_{L^1}+\|u_0\|_{L^2})
+C\int_0^t (1+t-\tau)^{-\frac{5}{4}}(\|G_2\|_{L^1}+\|G_2\|_{L^2})d\tau\\
\le
& C(\d_0+\d_1)(1+t)^{-\frac{5}{4}}
  +C\d_0\int_0^t(1+t-\tau)^{-\frac{5}{4}}(1+\tau)^{-\frac{5}{4}}d\tau\\
\le
& C(\d_0+\d_1)(1+t)^{-\frac{5}{4}},
\end{aligned}
\end{equation*}
which, together with estimate \eqref{276} and smallness of $\d_i(i=0,1)$, yields directly
\begin{equation*}
\|{\rm div} u\|_{L^2}
\ge C(1+t)^{-\frac{5}{4}}-C(\d_0+\d_1)(1+t)^{-\frac{5}{4}}\ge C(1+t)^{-\frac{5}{4}}.
\end{equation*}
This and the estimate \eqref{275} yield
\begin{equation*}
\|\p_t \vr\|_{L^2} \ge C(1+t)^{-\frac{5}{4}}-C(1+t)^{-\frac{5}{2}},
\end{equation*}
which implies directly
\begin{equation*}
\|\p_t \vr\|_{L^2} \ge C(1+t)^{-\frac{5}{4}}
\end{equation*}
for some large time $t$. Therefore, we complete the proof of this lemma.
\end{proof}

\section{Decay Estimates in Weighted Space}\label{space-time}

In this section, we will establish the decay rate of solution for the compressible MHD equation
\eqref{eq-per2} in weighted Sobolev space.
First of all, decay rates for the density, velocity and magnetic field in weighted norm $L^2_\g$
are established based on the technique lemma developed in  \cite{Kukavica-Torres-07} and \cite{Weng-JFA}.
Furthermore, we also address the decay rate for the higher order spatial derivatives in weighted norm.
To achieve this target, the Fourier splitting method developed by Schonbek \cite{Schonbek}
is used to establish the decay rate for the $(k+1)-$th order derivative if the decay rate for
$k-$th order derivative has been established.
Finally, we also address the lower bound of decay rate for the magnetic field in weighted norm.
This will show that the sharp decay rate of magnetic field converging to zero
in $L^2_\g$ is $t^{-\frac{3}{4}+\frac{\g}{2}}$.

Now, we state the following lemma, which can be found in \cite{Kukavica-Torres-07} and \cite{Weng-JFA}.
\begin{lemm}\label{space-time-lemma}
Let $\alpha_0>1, \alpha_1<1, \alpha_2 <1$, and $\beta_1<1, \beta_2<2$.
Assume that a continuously differential function $F: [1, \infty) \rightarrow [0, \infty)$ satisfies
\begin{equation*}
\begin{aligned}
&\frac{d}{dt}F(t)\le C_0 t^{-\alpha_0}F(t)+C_1 t^{-\alpha_1}F(t)^{\beta_1}
                     +C_2 t^{-\alpha_2}F(t)^{\beta_2}+C_3 t^{\g_2-1}, t\ge 1;\\
&F(1)\le K_0,\\
\end{aligned}
\end{equation*}
where $C_0, C_1, C_2, C_3, K_0>0$ and $\g_i=\frac{1-\alpha_i}{1-\beta_i}>0$ for $i=1,2$.
Assume that $\g_1 \ge \g_2$, then there exists a constant $C^*$ depending on
$\alpha_0, \alpha_1, \beta_1, \alpha_2, \beta_2, K_0, C_i, i=1,2,3,4$, such that
$$F(t)\le C^* t^{\g_1},$$
for all $ t\ge 1$.
\end{lemm}

Now, we address the space-time decay rate for the density, velocity and magnetic field of compressible
Hall-MHD equation \eqref{eq-per2}. More precisely, we have

\begin{lemm}
Under the assumptions of Theorem \ref{THM4}, then the smooth solution $(\vr, u, B)$ of compressible
Hall-MHD equation \eqref{eq-per2} has the estimates
\begin{gather}\label{291}
\|B(t)\|_{L^2_\g}\le Ct^{-\frac{3}{4}+\frac{\g}{2}},\\
\label{292}
\|\vr(t)\|_{L^2_\g}+\|u(t)\|_{L^2_\g} \le Ct^{-\frac{3}{4}+\g},
\end{gather}
for all $\g \ge 0$. Here $C$ is a positive constant independent of time.
\end{lemm}

\begin{proof}

The estimate \eqref{291} can be obtained just following the idea as in \cite{Weng-JFA}.
Now, we hope to establish the estimate \eqref{292}.
First of all, multiplying the first equation in \eqref{eq-per2}
by $\x^{2\g} \vr$ and integrating by part, we have
\begin{equation*}
\frac{d}{dt}\frac{1}{2}\int \x^{2\g}\vr^2 dx+\int \x^{2\g} \vr {\rm div} u dx
=\int \x^{2\g}\vr \cdot G_1 dx.
\end{equation*}
Using the H\"{o}lder, Cauchy and Sobolev inequalities, we have
\begin{equation}\label{297}
\begin{aligned}
|\int u \cdot \n \vr \cdot \x^{2\g}\vr dx|
&\le \|\n \vr\|_{L^3}\|\vr\|_{L^2_\g}\|u\|_{L^6_\g}\\
&\le \varepsilon \|\n u\|_{L^2_\g}^2+C\|u\|_{L^2_{\g-1}}^2+C\|\n \vr\|_{H^1}^2\|\vr\|_{L^2_\g}^2,
\end{aligned}
\end{equation}
where we have used the estimate
\begin{equation*}
\|u\|_{L^6_\g}\le C\|\n (\x^\g u)\|_{L^2}\le C (\|\n u\|_{L^2_\g}+\|u\|_{L^2_{\g-1}}).
\end{equation*}
Integrating by part, it holds on
\begin{equation}\label{298}
\begin{aligned}
&|\int \vr {\rm div} u \cdot \x^{2\g}\vr dx|
=|\int u \cdot \n(\x^{2\g}\vr^2) dx|\\
\le
&C\|\n \vr\|_{L^3}\|\vr\|_{L^2_\g}\|u\|_{L^6_\g}
 +C\|\vr\|_{L^\infty}\|\vr\|_{L^2_\g}\|u\|_{L^2_{\g-1}}\\
\le
&\varepsilon \|\n u\|_{L^2_\g}^2+C\|u\|_{L^2_{\g-1}}^2+C\|\n \vr\|_{H^1}^2\|\vr\|_{L^2_\g}^2.
\end{aligned}
\end{equation}
Combining the estimates \eqref{297} and \eqref{298}, we get
\begin{equation*}
|\int \x^{2\g}\vr \cdot G_1 dx|
\le \varepsilon\|\n u\|_{L^2_\g}^2+C\|u\|_{L^2_{\g-1}}^2+C\|\n \vr\|_{H^1}^2\|\vr\|_{L^2_\g}^2,
\end{equation*}
which implies directly
\begin{equation}\label{299}
\frac{d}{dt}\frac{1}{2}\int \x^{2\g}\vr^2 dx+\int \x^{2\g} \vr {\rm div} u dx
\le \varepsilon \|\n u\|_{L^2_\g}^2+C\|u\|_{L^2_{\g-1}}^2+C\|\n \vr\|_{H^1}^2\|\vr\|_{L^2_\g}^2.
\end{equation}

Next, we multiply the second equation in \eqref{eq-per2} by $\x^{2\g}u$ and integrate by part to obtain
\begin{equation*}\label{293}
\begin{aligned}
&\frac{d}{dt}\frac{1}{2}\int \x^{2\g}|u|^2 dx
+\mu \int \x^{2\g}|\n u|^2 dx
+(\mu+\nu)\int \x^{2\g}|{\rm div} u|^2 dx\\
&-\int \x^{2\g} \vr {\rm div}u dx=\int \x^{2\g}G_2 \cdot u dx+I_1+I_2.
\end{aligned}
\end{equation*}
where the functions $I_1$ and $I_2$ are defined by
$$
I_1:=2\g \int \x^{2\g-2}x_j \p_i x_j \vr u_i dx,
$$
and
$$
I_2:=-2\g (\mu+\nu)\int \x^{2\g-2}x_k \p_i x_k u_i \p_j u_j dx
    -2\g \mu \int \x^{2\g-2}x_k \p_j x_k \p_j u_i u_i dx.
$$
Using the H\"{o}lder and Cauchy inequalities, we have
\begin{equation*}
|I_1|
\le \|\vr\|_{L^2_\g}\|u\|_{L^2_{\g-1}}.
\end{equation*}
and
\begin{equation*}
\begin{aligned}
|I_2|
&\le C\mu \|\n u\|_{L^2_\g}\|u\|_{L^2_{\g-1}}
     +C(\mu+\nu) \|{\rm div} u\|_{L^2_\g}\|u\|_{L^2_{\g-1}} \\
&\le \var \mu \|\n u\|_{L^2_\g}^2+\var (\mu+\nu)\|{\rm div} u\|_{L^2_\g}^2
         +C\|u\|_{L^2_{\g-1}}^2.
\end{aligned}
\end{equation*}
Now, we deal with the term $\int \x^{2\g}G_2 \cdot u dx$.
Integrating by part, it holds on
\begin{equation*}
\begin{aligned}
&|\int \frac{\vr}{1+\vr}\Delta u \cdot \x^{2\g}u dx|
=|\int \n u \cdot \n (\x^{2\g}u \frac{\vr}{1+\vr})dx|\\
&\le C\|\vr\|_{L^\infty}\|\n u\|_{L^2_\g}^2
     +C\|\n \vr\|_{L^3}\|\n u\|_{L^2_\g}\|u\|_{L^6_\g}
     +C\|\vr\|_{L^\infty}\|\n u\|_{L^2_\g}\|u\|_{L^2_{\g-1}}\\
&\le \varepsilon \|\n u\|_{L^2_\g}^2
     +C\|u\|_{L^2_{\g-1}}^2.
\end{aligned}
\end{equation*}
Similarly, we also get
\begin{equation*}
|\int \frac{\vr}{1+\vr} \n {\rm div}u \cdot \x^{2\g}u dx|
\le \varepsilon\|{\rm div} u\|_{L^2_\g}^2+C\|u\|_{L^2_{\g-1}}^2.
\end{equation*}
Similar to the estimate \eqref{297}, it holds on
\begin{equation*}
\begin{aligned}
&|\int u \cdot \n u \cdot \x^{2\g}udx|
\le \varepsilon \|\n u\|_{L^2_\g}^2+C\|u\|_{L^2_{\g-1}}^2+C\|\n u\|_{H^1}^2\|u\|_{L^2_\g}^2,\\
&|\int \left\{\frac{P'(\vr+1)}{\vr+1}-1\right\}\n \vr \cdot \x^{2\g}udx|
\le \varepsilon \|\n u\|_{L^2_\g}^2+C\|u\|_{L^2_{\g-1}}^2+C\|\n \vr\|_{H^1}^2\|\vr\|_{L^2_\g}^2,\\
&|\int \frac{(\nt B)\times B }{1+\vr} \cdot \x^{2\g}udx|
\le \varepsilon \|\n u\|_{L^2_\g}^2+C\|u\|_{L^2_{\g-1}}^2+C\|\n B\|_{H^1}^2\| B\|_{L^2_\g}^2.
\end{aligned}
\end{equation*}
Thus, we use the smallness of $\varepsilon$ to get
\begin{equation*}\label{293}
\begin{aligned}
&\frac{d}{dt}\frac{1}{2}\int \x^{2\g}|u|^2 dx
+\frac{3\mu}{4} \int \x^{2\g}|\n u|^2 dx
+\frac{3(\mu+\nu)}{4}\int \x^{2\g}|{\rm div} u|^2 dx\\
\le
& C\|\n (\vr, u)\|_{H^1}^2 \|(\vr, u)\|_{L^2_\g}^2
  +C\|\n B\|_{H^1}^2\|B\|_{L^2_\g}^2
  +C\|\vr\|_{L^2_\g}\|u\|_{L^2_{\g-1}}
  +C\|u\|_{L^2_{\g-1}}^2,
\end{aligned}
\end{equation*}
which, together with \eqref{299}, yields directly
\begin{equation}\label{296}
\begin{aligned}
&\frac{d}{dt}\int \x^{2\g}(|\vr|^2+|u|^2) dx
+\int \x^{2\g}(\mu |\n u|^2+(\mu+\nu)|{\rm div} u|^2) dx\\
\le &C \|\n (\vr, u)\|_{H^1}^2 \|(\vr, u)\|_{L^2_\g}^2
     \!+C\|\n B\|_{H^1}^2\|B\|_{L^2_\g}^2
     \!+C\|\vr\|_{L^2_\g}\|u\|_{L^2_{\g-1}}
     \!+C\|u\|_{L^2_{\g-1}}^2.
\end{aligned}
\end{equation}
Using the inequality
$
\|u\|_{L^2_{\g-1}} \le \|u\|_{L^2_\g}^{\frac{\g-1}{\g}}\|u\|_{L^2}^{\frac{1}{\g}},
$
then we have
\begin{equation*}
\begin{aligned}
&\frac{d}{dt}\int \x^{2\g}(|\vr|^2+|u|^2) dx
+\int \x^{2\g}(\mu |\n u|^2+(\mu+\nu)|{\rm div} u|^2) dx\\
\le &C \|\n (\vr, u, B)\|_{H^1}^2 \|(\vr, u)\|_{L^2_\g}^2
     \!+C\|\n B\|_{H^1}^2\|B\|_{L^2_\g}^2
     \!+\!\|\vr\|_{L^2_\g}\|u\|_{L^2_\g}^{\frac{\g-1}{\g}}\|u\|_{L^2}^{\frac{1}{\g}}
     \!+C\|u\|_{L^2_\g}^{\frac{2(\g-1)}{\g}}\|u\|_{L^2}^{\frac{2}{\g}},
\end{aligned}
\end{equation*}
Denoting $E(t):=\|\vr (t)\|_{L^2_\g}^2+\|u(t)\|_{L^2_\g}^2$,
we can obtain
\begin{equation}\label{294}
\frac{d}{dt}E(t)
\le C_0 t^{-\frac{5}{4}} E(t)
    +C_1 t^{-\frac{3}{4 \g}}E(t)^{\frac{2\g-1}{2\g}}
    +C_2 t^{-\frac{3}{2 \g}}E(t)^{\frac{\g-1}{\g}}
    +C_3 t^{-\frac{11}{4}+\g}.
\end{equation}
Here $\alpha_0=\frac{5}{4}, \alpha_1=\frac{3}{4 \g}, \beta_1=\frac{2\g-1}{2\g},
\alpha_2=\frac{3}{2 \g}, \beta_2=\frac{\g-1}{\g}$.
To assure that $\alpha_1<1, \alpha_2<1$, we require $\g >\frac{3}{2}$.
Hence, $\lambda_1=\frac{1-\alpha_1}{1-\beta_1}=2\g-\frac{3}{2},
 \lambda_2=\frac{1-\alpha_2}{1-\beta_2}=\g-\frac{3}{2}$,
and hence we have $\lambda_1 >\lambda_2$. Thus, using the Lemma \ref{space-time-lemma},
we can deduce from \eqref{294} that
\begin{equation}\label{295}
E(t) \le Ct^{\lambda_1}=Ct^{-\frac{3}{2}+2\g}.
\end{equation}
Using the interpolation inequality and estimate \eqref{295}, we have
$$
\|\vr(t)\|_{L^2_{\g_0}}
\le C \| \vr(t)\|_{L^2}^{1-\frac{\g_0}{\g}}
      \| \vr(t)\|_{L^2_\g}^{\frac{\g_0}{\g}}
\le C t^{-\frac{3}{4}+\g_0},
$$
for all $\g_0 \in [0, \g]$.
The estimate for the velocity  can be obtained similarly,
and hence, we complete the proof of lemma.
\end{proof}

Next, we hope to establish decay rate for the spatial derivative of magnetic field
in weight Sobolev space. It should be pointed out that this target for the incompressible Hall-MHD equation
has been achieved in \cite{{Weng-JFA}} by using the parabolic interpolation inequality developed by
Kukavica and Torres \cite{{Kukavica-01},{Kukavica-Torres-06}}.
In the sequence, we will get that by using the Fourier Splitting method introduced by Schonbek \cite{Schonbek}.

\begin{lemm}
Under all the assumptions of Theorem \ref{THM4}, it holds on
\begin{equation}\label{2101}
\|B(t)\|_{H^{3}_\g}\le Ct^{-\frac{3}{4}+\frac{\g}{2}},
\end{equation}
for all $\g \ge 0$. Here $C$ is a positive constant independent of time.
\end{lemm}

\begin{proof}
Similar to the estimate \eqref{296}, it is easy to check that
\begin{equation}\label{2102}
\frac{d}{dt}\|B\|_{L^2_\g}^2+\|\n B\|_{L^2_\g}^2
\le C\|B\|_{L^2_{\g-1}}^2+C\|u\|_{L^\infty}^2\|B\|_{L^2_\g}^2.
\end{equation}
Now, we claim the following estimate, which will be proved in section \ref{technical}, holds on
\begin{equation}\label{claim4}
\begin{aligned}
&\frac{d}{dt}\|\n^k B\|_{H^{3-k}_\g}^2+\|\n^{k+1} B\|_{H^{3-k}_\g}^2\\
\le
&   C\|\n^k B\|_{H^{3-k}_{\g-1}}^2
    +C\|u\|_{L^\infty}^2\|\n^k B\|_{H^{3-k}_\g}^2\\
&   +C\|\n^2 (\vr, u, B)\|_{H^1}^2
     (\|\n B\|_{H^2_\g}^2+\|B\|_{H^2_{\g-1}}^2+\|B\|_{L^2_{\g-2}}^2),
\end{aligned}
\end{equation}
where $k=1,2,3$. Using \eqref{2102}, \eqref{claim4}, uniform estimate \eqref{uniform},
and decay rate \eqref{Decay1}, we have
\begin{equation*}
\frac{d}{dt}\| B\|_{H^{3}_\g}^2+\|\n B\|_{H^{3}_\g}^2
\le  C\|B\|_{H^{3}_{\g-1}}^2
     +C \d_0 \|B\|_{H^{3}_\g}^2
     +C t^{-\frac{7}{2}} \|B\|_{L^2_{\g-2}}^2.
\end{equation*}
Adding both sides of the above inequality with $\|B\|_{L^2_\g}^2$,
and using the smallness of $\d_0$, it holds on
\begin{equation*}
\frac{d}{dt}\| B\|_{H^{3}_\g}^2+\|B\|_{H^{4}_\g}^2
\le  C\|B\|_{H^{3}_{\g-1}}^2+\|B\|_{L^2_\g}^2
     +C t^{-\frac{7}{2}} \|B\|_{L^2_{\g-2}}^2
\le C\|B\|_{H^{3}_{\g-1}}^2+Ct^{-\frac{3}{2}+\g},
\end{equation*}
where we have used \eqref{291} in the last inequality.
Then, we can use the Gronwall inequality to get
\begin{equation}\label{2103}
\|B(t)\|_{H^3_\g}^2
\le \|B(T_*)\|_{H^3_\g}^2 e^{-(t-T_*)}
    +C\int_{T_*}^t e^{-(t-\tau)}(\|B(\tau)\|_{H^3_{\g-1}}^2+\tau^{-\frac{3}{2}+\g})d\tau.
\end{equation}
Taking $\g=1$ in \eqref{2103}, we use the decay rate \eqref{Decay1} to obtain
\begin{equation}\label{2104}
\begin{aligned}
\|B(t)\|_{H^3_1}^2
&\le \|B(T_*)\|_{H^3_1}^2 e^{-(t-T_*)}
    +C\int_{T_*}^t e^{-(t-\tau)}(\|B(\tau)\|_{H^3}^2+\tau^{-\frac{3}{2}+1})d\tau\\
&\le \|B(T_*)\|_{H^3_1}^2 e^{-(t-T_*)}
    +C\int_{T_*}^t e^{-(t-\tau)}(\tau^{-\frac{3}{2}}+\tau^{-\frac{1}{2}})d\tau\\
&\le Ct^{-\frac{1}{2}}.
\end{aligned}
\end{equation}
Now, we will take the strategy of induction to give the proof for estimate \eqref{2101}.
In fact, the estimate \eqref{2104} implies that \eqref{2101} holds on for the case $\g=1$.
By the general step of induction, assume that the estimate \eqref{2101} holds on for
the case $\g=\g_1 \ge 1$.  Then, we need to verify that estimate \eqref{2101} holds on for the case $\g=\g_1+1$.
Then, taking $\g=\g_1+1$ in estimate \eqref{2103}, we get
\begin{equation*}
\begin{aligned}
\|B(t)\|_{H^3_{\g_1+1}}^2
&\le \|B(T_*)\|_{H^3_{\g_1+1}}^2 e^{-(t-T_*)}
    +C\int_{T_*}^t e^{-(t-\tau)}(\|B(\tau)\|_{H^3_{\g_1}}^2+\tau^{-\frac{3}{2}+\g_1+1})d\tau\\
&\le \|B(T_*)\|_{H^3_{\g_1+1}}^2 e^{-(t-T_*)}
    +C\int_{T_*}^t e^{-(t-\tau)}(\tau^{-\frac{3}{2}+\g_1}+\tau^{-\frac{3}{2}+\g_1+1})d\tau\\
&\le Ct^{-\frac{1}{2}+\g_1}.
\end{aligned}
\end{equation*}
Hence, we have verified that \eqref{2101} holds on for the case $\g=\g_1+1$.
By the general step of induction, we complete the proof of lemma.
\end{proof}

Now, we will establish the optimal decay rate for higher order spatial derivative
of magnetic field in weighted norm by
using the Fourier Splitting method introduced by Schonbek \cite{Schonbek}.

\begin{lemm}
Under all the assumptions of Theorem \ref{THM4}, it holds on
\begin{equation}\label{2111}
\|\n^k B(t)\|_{H^{3-k}_\g}\le Ct^{-\frac{3}{4}+\frac{\g}{2}-\frac{k}{2}},
\end{equation}
for all $\g \ge 0$, and $k=0,1,2,3$. Here $C$ is a positive constant independent of time.
\end{lemm}

\begin{proof}
The proof of \eqref{2111} is only given for the integer $\g$, and the other case
can be obtained just by using the interpolation inequality.
We will take the strategy of induction to give the proof for estimate \eqref{2111}.
In fact, the inequality \eqref{2101} implies that \eqref{2111} holds on for the case $k=0$.
By the general step of induction, assume that the estimate \eqref{2111} holds on for
$k=\l(l=0,1,2)$, i.e.,
\begin{equation}\label{2112}
\|\n^l B(t)\|_{H^{3-l}_\g}\le Ct^{-\frac{3}{4}+\frac{\g}{2}-\frac{l}{2}}.
\end{equation}
Then, we need to verify
\begin{equation}\label{21110}
\|\n^{l+1} B(t)\|_{H^{2-l}_\g}\le Ct^{-\frac{3}{4}+\frac{\g}{2}-\frac{l+1}{2}}
\end{equation}
holding on. Taking $k=l+1$ in \eqref{claim4} and using decay rate \eqref{Decay1}, it holds on
\begin{equation}\label{2113}
\begin{aligned}
&\frac{d}{dt}\|\n^{l+1} B\|_{H^{2-l}_\g}^2+\|\n^{l+2} B\|_{H^{2-l}_\g}^2\\
\le
&   C\|\n^{l+1} B\|_{H^{2-l}_{\g-1}}^2
    +C t^{-3}\|\n^{l+1} B\|_{H^{2-l}_\g}^2\\
&   +C t^{-\frac{7}{2}}
     (\|\n B\|_{H^2_\g}^2+\|B\|_{H^2_{\g-1}}^2+\|B\|_{L^2_{\g-2}}^2).
\end{aligned}
\end{equation}
For $t>0$, denote the time sphere
$$
S_0:=\left\{\left. \xi\in \mathbb{R}^3\right| |\xi|\le \left(\frac{R}{t}\right)^\frac{1}{2}\right\},
$$
then we apply the Fourier-Plancheral formula to obtain
\begin{equation}\label{2114}
\begin{aligned}
\int_{\mathbb{R}^3} |\x^\g \n^{l+2} B|^2 dx
&=\int_{\mathbb{R}^3} |x|^{2\g} |\n^{l+2} B|^2 dx
=\sum_{|\alpha|=\g}\int_{\mathbb{R}^3} |x^\alpha \n^{l+2} B|^2 dx\\
&=\int_{\mathbb{R}^3}|\p_\xi^\g (\xi^{l+2} \f{B})|^2 d\xi
 \ge \int_{\mathbb{R}^3/{S_0}}|\p_\xi^\g (\xi^{l+2} \f{B})|^2 d\xi.
\end{aligned}
\end{equation}
Due to the fact
$\p_\xi^\g (\xi^{l+2} \f{B})
=\g \p_\xi^{\g-1} (\xi^{l+1} \f{B})+\xi \p_\xi^\g(\xi^{l+1} \f{B})$,
then it holds on
\begin{equation}\label{2115}
\begin{aligned}
\int_{\mathbb{R}^3/{S_0}}\!\!\!|\p_\xi^\g (\xi^{l+2} \f{B})|^2 d\xi
&\ge \int_{\mathbb{R}^3/{S_0}}|\xi \p_\xi^\g(\xi^{l+1} \f{B})
                               +\g \p_\xi^{\g-1} (\xi^{l+1} \f{B})|^2 d\xi\\
&\ge \frac{1}{2}\int_{\mathbb{R}^3/{S_0}}|\xi \p_\xi^\g(\xi^{l+1} \f{B})|^2 d\xi
     -\int_{\mathbb{R}^3/{S_0}}|\g \p_\xi^{\g-1} (\xi^{l+1} \f{B})|^2 d\xi\\
&\ge \frac{R}{2t}\int_{\mathbb{R}^3/{S_0}}|\p_\xi^\g(\xi^{l+1} \f{B})|^2 d\xi
     -\int_{\mathbb{R}^3/{S_0}}\g^2|\p_\xi^{\g-1} (\xi^{l+1} \f{B})|^2 d\xi\\
&\ge \frac{R}{2t}\int_{\mathbb{R}^3}\!|\p_\xi^\g(\xi^{l+1} \f{B})|^2 d\xi
     -\frac{R}{2t}\int_{S_0}\!|\p_\xi^\g(\xi^{l+1} \f{B})|^2 d\xi
     -\g^2\!\! \int_{\mathbb{R}^3/{S_0}}\!\!\!|\p_\xi^{\g-1} (\xi^{l+1} \f{B})|^2 d\xi.
\end{aligned}
\end{equation}
It is easy to check that
$$
|\p_\xi^\g (\xi^{l+1} \f{B})|^2
=|\xi \p_\xi^\g(\xi^l \f{B})+\g \p_\xi^{\g-1} (\xi^l \f{B})|^2
\le 2 |\xi|^2 |\p_\xi^\g (\xi^{l}\f{B})|^2+2 \g^2  |\p_\xi^{\g-1} (\xi^l \f{B})|^2,
$$
 hence, we have
\begin{equation}\label{2116}
\begin{aligned}
\int_{S_0}|\p_\xi^\g(\xi^{l+1} \f{B})|^2 d\xi
&\le 2 \int_{S_0}|\xi|^2 |\p_\xi^\g (\xi^l \f{B})|^2 d\xi
   +2 \g^2  \int_{S_0} |\p_\xi^{\g-1} (\xi^l \f{B})|^2 d\xi\\
&\le \frac{2R}{t} \int_{S_0}|\p_\xi^\g (\xi^l \f{B})|^2 d\xi
   +2 \g^2  \int_{S_0} |\p_\xi^{\g-1} (\xi^l \f{B})|^2 d\xi.
\end{aligned}
\end{equation}
Combining the inequalities \eqref{2114}, \eqref{2115} and \eqref{2116}, one arrives at
\begin{equation}\label{2117}
\|\n^{l+2} B(t)\|_{L^2_\g}^2
\ge \frac{R}{2t}\|\n^{l+1} B(t)\|_{L^2_\g}^2-\frac{R^2}{t^2}\|\n^{l} B(t)\|_{L^2_\g}^2
    -\frac{R\g^2}{t}\|\n^{l} B(t)\|_{L^2_{\g-1}}^2
    -\g^2 \|\n^{l+1} B(t)\|_{L^2_{\g-1}}^2.
\end{equation}
Similarly, it is easy to deduce that
\begin{equation}\label{2118}
\begin{aligned}
\|\n^{l+2} B(t)\|_{H^{2-l}_\g}^2
\ge
&  \frac{R}{2t}\|\n^{l+1} B(t)\|_{H^{2-l}_\g}^2
    -\frac{R^2}{t^2}\|\n^{l} B(t)\|_{H^{2-l}_\g}^2\\
&   -\frac{R\g^2}{t}\|\n^{l} B(t)\|_{H^{2-l}_{\g-1}}^2
    -\g^2 \|\n^{l+1} B(t)\|_{H^{2-l}_{\g-1}}^2,
\end{aligned}
\end{equation}
which, together with \eqref{2113}, yields directly
\begin{equation*}
\begin{aligned}
&\frac{d}{dt}\|\n^{l+1} B(t)\|_{H^{2-l}_\g}^2
  +\frac{R}{2t}\|\n^{l+1} B(t)\|_{H^{2-l}_\g}^2\\
\le
&   \frac{R^2}{t^2}\|\n^{l} B(t)\|_{H^{2-l}_\g}^2
    +\frac{R\g^2}{t}\|\n^{l} B(t)\|_{H^{2-l}_{\g-1}}^2
    +(1+\g^2) \|\n^{l+1} B(t)\|_{H^{2-l}_{\g-1}}^2\\
&   +C t^{-3}\|\n^{l+1} B\|_{H^{2-l}_\g}^2
    +C t^{-\frac{7}{2}}(\|\n B\|_{H^2_\g}^2+\|B\|_{H^2_{\g-1}}^2+\|B\|_{L^2_{\g-2}}^2)\\
\le
& C t^{-\frac{7}{2}+\g-l}+C\|\n^{l+1} B(t)\|_{H^{2-l}_{\g-1}}^2,
\end{aligned}
\end{equation*}
where we have used the assumption decay rate \eqref{2112}.
Then, for some large time $t$, it holds on
\begin{equation}\label{2119}
\frac{d}{dt}\|\n^{l+1} B(t)\|_{H^{2-l}_\g}^2
  +\frac{R}{2t}\|\n^{l+1} B(t)\|_{H^{2-l}_\g}^2
\le C t^{-\frac{7}{2}+\g-l}+C\|\n^{l+1} B(t)\|_{H^{2-l}_{\g-1}}^2.
\end{equation}
If $\g =1$, taking $R=2(l+2)$ in the above inequality, one arrives at
\begin{equation*}
\frac{d}{dt}\|\n^{l+1} B(t)\|_{H^{2-l}_1}^2
  +\frac{l+2}{t}\|\n^{l+1} B(t)\|_{H^{2-l}_1}^2
\le C t^{-\frac{5}{2}-l}+C\|\n^{l+1} B(t)\|_{H^{2-l}}^2
\le C t^{-\frac{5}{2}-l}.
\end{equation*}
which, multiplying by $t^{2+l}$ and integrating with respect with time, we have
\begin{equation}\label{21111}
\|\n^{l+1} B(t)\|_{H^{2-l}_1}^2 \le C t^{-\frac{3}{2}-l}.
\end{equation}
Now, We will take the strategy of induction to give the proof for estimate  \eqref{21110}.
In fact, the estimate \eqref{21111} implies \eqref{21110} holding on for the case $\g=1$.
By the general step of induction, assume that the estimate \eqref{21111} holds on for
the case $\g=\g_1 \ge 1$.  Then, we need to verify that \eqref{21111} holds on for the case $\g=\g_1+1$.
Indeed, taking $R=2(l+2)$ in inequality \eqref{2119}, it holds on
\begin{equation*}
\frac{d}{dt}\|\n^{l+1} B(t)\|_{H^{2-l}_{\g_1+1}}^2
  +\frac{l+2}{t}\|\n^{l+1} B(t)\|_{H^{2-l}_{\g_1+1}}^2
\le C t^{-\frac{7}{2}+\g_1+1-l}+C\|\n^{l+1} B(t)\|_{H^{2-l}_{\g_1}}^2
\le C t^{-\frac{5}{2}+\g_1-l}.
\end{equation*}
Then, multiplying the above inequality by $t^{l+2}$
and integrating with respect to time, we obtain
\begin{equation*}
\|\n^{l+1} B (t)\|_{H^{2-l}_{\g_1+1}}^2 \le C t^{-\frac{3}{2}+\g_1-l}.
\end{equation*}
Hence, we have verified that \eqref{21110} holds on for the case $\g=\g_1+1$.
By the general step of induction, we have verified that estimate \eqref{21110}
holds on. Then, we complete the proof of lemma due to the general step of induction.
\end{proof}

Finally, we concentrate on the space-time decay rate for the higher order spatial
derivative of density and velocity.

\begin{lemm}
Under all the assumptions of Theorem \ref{THM4}, it holds on
\begin{equation}\label{2121}
\|\vr(t)\|_{H^{3}_\g}+\|u(t)\|_{H^{3}_\g}\le Ct^{-\frac{3}{4}+\g},
\end{equation}
for all $\g \ge 0$. Here $C$ is a positive constant independent of time.
\end{lemm}

\begin{proof}
First of all, we claim the following estimate, which will be proved in section \ref{technical}, holds on
\begin{equation}\label{claim5}
\begin{aligned}
&\frac{d}{dt}\c{E}^{3,k}_\g(t)
  +C(\|\n^{k+1} \vr\|_{H^{2-k}_\g}^2+\|\n^{k+1} u\|_{H^{3-k}_\g}^2) \\
\le
& C\|\n^k (\vr, u)\|_{H^{3-k}_{\g-1}}^2+C\|\n^k \vr\|_{H^{3-k}_\g}\|\n^k u\|_{H^{3-k}_{\g-1}}
  +C\|\n (\vr, u, B)\|_{H^2}^2 \|\n (\vr, u, B)\|_{H^{2}_\g}^2,
\end{aligned}
\end{equation}
for $k=0,1,2$.
Here the energy $\c{E}^{3,k}_\g(t)$ is defined by
\begin{equation}
\c{E}^{3,k}_\g(t):=\|\n^k \vr(t)\|_{H^{3-k}_\g}^2+\|\n^k u(t)\|_{H^{3-k}_\g}^2
                    +C\d_0 \sum_{k\le l \le 2}\int  \x^{2\g} \n^l u \cdot \n^{l+1}\vr dx.
\end{equation}
Due to the smallness of $\d_0$,
there are two constants $C_1$ and $C_2$ such that
\begin{equation}\label{equivalent2}
C_1(\|\n^k \vr(t)\|_{H^{3-k}_\g}^2+\|\n^k u(t)\|_{H^{3-k}_\g}^2)
\le \c{E}^{3,k}_\g(t)
\le C_2(\|\n^k \vr(t)\|_{H^{3-k}_\g}^2+\|\n^k u(t)\|_{H^{3-k}_\g}^2).
\end{equation}
Taking $k=1$ in \eqref{claim5}, adding with \eqref{296},
and adding with $\|(\vr, u)\|_{L^2_\g}^2$ in both handsides, we have
\begin{equation}
\frac{d}{dt}\c{E}^{3,0}_\g(t)
  +C(\|\vr\|_{H^{3}_\g}^2+\|u\|_{H^{4}_\g}^2)
\le C\|(\vr, u)\|_{H^{3}_{\g-1}}^2+C\|(\vr, u)\|_{L^2_\g}^2,
\end{equation}
where we have used the Cauchy inequality and uniform estimate \eqref{uniform}.
Then, using equivalent relation \eqref{equivalent2}, we can obtain
\begin{equation*}
\frac{d}{dt}\c{E}^{3,0}_\g(t)+C \c{E}^{3,0}_\g(t)
\le C\|(\vr, u)\|_{H^{3}_{\g-1}}^2+C\|(\vr, u)\|_{L^2_\g}^2.
\end{equation*}
Thus, similar to the estimate \eqref{2101}, we can apply the induction with respect to
the exponent $\g$ to establish the estimate \eqref{2121}. Therefore, we complete the proof of lemma.
\end{proof}

\begin{lemm}
Under all the assumptions of Theorem \ref{THM4}, it holds on
\begin{equation}\label{2131}
\|\n^k \vr(t)\|_{H^{3-k}_\g}
+\|\n^k u(t)\|_{H^{3-k}_\g} \le Ct^{-\frac{3}{4}+\g-\frac{k}{2}},
\end{equation}
for all $\g \ge 0$, and $k=0,1, 2$. Here $C$ is a positive constant independent of time.
\end{lemm}

\begin{proof}
The proof of \eqref{2131} is only given for the integer $\g$, and the other case
can be obtained just by using the interpolation.
We will take the strategy of induction to give the proof for estimate \eqref{2131}.
In fact, the inequality \eqref{2121} implies that \eqref{2131} holds on for the case $k=0$.
By the general step of induction, assume that the estimate \eqref{2131} holds on for
the case $k=\l(l=0,1)$, i.e.,
\begin{equation}\label{2132}
\|\n^l \vr(t)\|_{H^{3-l}_\g}+\|\n^l u(t)\|_{H^{3-l}_\g}\le Ct^{-\frac{3}{4}+\g-\frac{l}{2}}.
\end{equation}
Then, we need to verify that \eqref{2131} holds on for the case $k=l+1$, i.e.,
\begin{equation}\label{2138}
\|\n^{l+1} \vr(t)\|_{H^{2-l}_\g}+\|\n^{l+l} u(t)\|_{H^{2-l}_\g}
\le Ct^{-\frac{3}{4}+\g-\frac{l+1}{2}}.
\end{equation}
Taking $k=l+1$ in the inequality \eqref{claim5} and using decay rate \eqref{Decay1}, then we have
\begin{equation*}
\begin{aligned}
&\frac{d}{dt}\c{E}^{3,l+1}_\g(t)
  +C(\|\n^{l+2} \vr\|_{H^{1-l}_\g}^2+\|\n^{l+2} u\|_{H^{2-l}_\g}^2) \\
\le
& C\|\n^{l+1}(\vr, u)\|_{H^{2-l}_{\g-1}}^2+C\|\n^{l+1} \vr\|_{H^{2-l}_\g}\|\n^{l+1} u\|_{H^{2-l}_{\g-1}}
  +C t^{-\frac{5}{2}} \|\n (\vr, u, B)\|_{H^{2}_\g}^2,
\end{aligned}
\end{equation*}
or equivalently, it holds on
\begin{equation*}
\begin{aligned}
&\frac{d}{dt}\c{E}^{3,l+1}_\g(t)
  +\frac{C}{2}\|\n^3 \vr\|_{L^2_\g}^2
  +\frac{C}{2} (\|\n^{l+2} \vr\|_{H^{1-l}_\g}^2+\|\n^{l+2} u\|_{H^{2-l}_\g}^2) \\
\le
& C\|\n^{l+1}(\vr, u)\|_{H^{2-l}_{\g-1}}^2+C\|\n^{l+1} \vr\|_{H^{2-l}_\g}\|\n^{l+1} u\|_{H^{2-l}_{\g-1}}
  +C t^{-\frac{5}{2}} \|\n (\vr, u, B)\|_{H^{2}_\g}^2.
\end{aligned}
\end{equation*}
Taking the Fourier Splitting method, similar to estimate \eqref{2117} or \eqref{2118}, we have
\begin{equation}\label{2133}
\begin{aligned}
&\frac{d}{dt}\c{E}^{3,l+1}_\g(t)
  +\frac{C}{2}\|\n^3 \vr\|_{L^2_\g}^2
  +\frac{CR}{4t} (\|\n^{l+1} \vr\|_{H^{1-l}_\g}^2+\|\n^{l+1} u\|_{H^{2-l}_\g}^2) \\
\le
& C t^{-2}(\|\n^{l} \vr\|_{H^{1-l}_\g}^2+\|\n^{l} u\|_{H^{2-l}_\g}^2)
   +C t^{-1} (\|\n^{l+1} \vr\|_{H^{1-l}_{\g-1}}^2+\|\n^{l+1} u\|_{H^{2-l}_{\g-1}}^2)\\
& +C(\|\n^{l+1} \vr\|_{H^{2-l}_{\g-1}}^2+\|\n^{l+1} u\|_{H^{2-l}_{\g-1}}^2)
  +C\|\n^{l+1} \vr\|_{H^{2-l}_\g}\|\n^{l+1} u\|_{H^{2-l}_{\g-1}}\\
&
  +C t^{-\frac{5}{2}} \|\n (\vr, u, B)\|_{H^{2}_\g}^2.
\end{aligned}
\end{equation}
For some large time $t$ such that $1\ge R/(2t)$, then it holds on
\begin{equation}\label{2134}
\frac{C}{2}\|\n^3 \vr\|_{L^2_\g}^2+\frac{CR}{4t} \|\n^{l+1} \vr\|_{H^{1-l}_\g}^2
\ge \frac{CR}{4t}(\|\n^3 \vr\|_{L^2_\g}^2+\|\n^{l+1} \vr\|_{H^{1-l}_\g}^2)
\ge \frac{CR}{4t}\|\n^{l+1} \vr\|_{H^{2-l}_\g}^2.
\end{equation}
Using the Cauchy inequality, we have
\begin{equation}\label{2135}
\|\n^{l+1} \vr\|_{H^{2-l}_\g}\|\n^{l+1} u\|_{H^{2-l}_{\g-1}}
\le \var t^{-1}\|\n^{l+1} \vr\|_{H^{2-l}_\g}^2
     +C t \|\n^{l+1} u\|_{H^{2-l}_{\g-1}}^2.
\end{equation}
Using assumption estimate \eqref{2132}, inequalities \eqref{2133}, \eqref{2134}, \eqref{2135},
and equivalent relation \eqref{equivalent2}, it holds on
\begin{equation}\label{2136}
\frac{d}{dt}\c{E}^{3,l+1}_\g(t)+\frac{CR}{4t} \c{E}^{3,l+1}_\g(t)
\le
 C\|\n^{l+1} \vr\|_{H^{2-l}_{\g-1}}^2+ Ct\|\n^{l+1} u\|_{H^{2-l}_{\g-1}}^2
  +C t^{-\frac{7}{2}+2\g-l},
\end{equation}
for large time $t$.
Taking $\g=1$ in \eqref{2136} and using decay rate \eqref{Decay1}, then we have
\begin{equation*}
\frac{d}{dt}\c{E}^{3,l+1}_1(t)+\frac{CR}{4t} \c{E}^{3,l+1}_1(t)
\le
 C\|\n^{l+1} \vr\|_{H^{2-l}}^2+ Ct\|\n^{l+1} u\|_{H^{2-l}}^2
  +C t^{-\frac{3}{2}-l}
\le C t^{-\frac{3}{2}-l}.
\end{equation*}
Choosing $R=\frac{4(l+1)}{C}$ in the above inequality
and multiplying by $t^{l+1}$, we obtain
\begin{equation*}
\frac{d}{dt}[t^{l+1} \c{E}^{3,l+1}_1(t)]\le Ct^{-\frac{1}{2}},
\end{equation*}
which, integrating with respect to time and using the equivalent relation \eqref{equivalent2}, yields directly
\begin{equation}\label{2137}
\|\n^{l+1} \vr(t)\|_{H^{2-l}_1}^2+\|\n^{l+1} u(t)\|_{H^{2-l}_1}^2
\le C t^{-\frac{1}{2}-l}.
\end{equation}
Now, we will take the strategy of induction to give the proof to estimate \eqref{2138}.
In fact, the decay rate \eqref{2137} implies \eqref{2138} holding on for the case $\g=1$.
By the general step of induction, assume that the estimate \eqref{2138} holds on for
the case $\g=\g_1 \ge 1$.  Then, we need to verify that \eqref{2138} holds on for the case $\g=\g_1+1$.
Thus, taking $\g=\g_1+1$ and $R=\frac{4(l+1)}{C}$ in \eqref{2136}, it holds on
\begin{equation*}
\frac{d}{dt}\c{E}^{3,l+1}_{\g_1+1}(t)
  +\frac{l+1}{t} \c{E}^{3,l+1}_{\g_1+1}(t)
\le C t^{-\frac{3}{2}+2\g_1-l}.
\end{equation*}
Then, multiplying the above inequality by $t^{l+1}$
and integrating with respect to time, we obtain
\begin{equation*}
\|\n^{l+1} \vr (t)\|_{H^{2-l}_{\g_1+1}}^2
+\|\n^{l+1} u (t)\|_{H^{2-l}_{\g_1+1}}^2\le C t^{-\frac{3}{2}+2(\g_1+1)-(l+1)},
\end{equation*}
which implies that \eqref{2138} holds on for the case $\g=\g_1+1$.
By the general step of induction, we have verified the estimate \eqref{2138}.
Then, we complete the proof of lemma due to the general step of induction.
\end{proof}

Therefore, we can obtain the decay estimates \eqref{Decay42} and \eqref{Decay44}
by using the decay rates \eqref{2111}, \eqref{2131} and equation \eqref{eq-per2}.
Finally, we address the lower decay rate for the solution of compressible
Hall-MHD equation \eqref{eq-per2} in weighted norm.
For any $s\in [0, 3/2)$, we using the Hardy and H\"{o}lder inequalities to get
$$
\|f\|_{L^2_{-s}}\le C\|f\|_{\dot{H}^{s}},
$$
and
$$
\|f\|_{L^2}\le \|f\|_{L^2_{s}}^{\frac{1}{2}} \|f\|_{L^2_{-s}}^{\frac{1}{2}},
$$
where $\dot{H}^{s}$ denotes the homogeneous Sobolev space and $f$ is a suitable function.
Using decay rate \eqref{Decay1}, then it holds on
\begin{equation*}
Ct^{-\frac{3}{4}}\le \|B(t)\|_{L^2}
\le C\|B(t)\|_{L^2_{\g}}^{\frac{1}{2}} \|B(t)\|_{L^2_{-\g}}^{\frac{1}{2}}
\le C\|B(t)\|_{L^2_{\g}}^{\frac{1}{2}}\|B(t)\|_{\dot{H}^{\g}}^{\frac{1}{2}}
\le C t^{-\frac{3}{8}-\frac{\g}{4}}\|B(t)\|_{L^2_{\g}}^{\frac{1}{2}},
\end{equation*}
for all $\g \in [0, 3/2)$.
If the Fourier transform $\widehat{B_0}=\mathcal{F}(B_0)$
satisfies $|\widehat{B_0}|\ge c_0$ for all $0\le |\xi| \ll 1$, then
it follows from decay rate \eqref{Decay23} that
$$\|B(t)\|_{L^2} \ge Ct^{-\frac{3}{4}}.$$
Thus, we can obtain the lower bound of decay rate estimate
\begin{equation}\label{2141}
\|B(t)\|_{L^2_\g}\ge C t^{-\frac{3}{4}+\frac{\g}{2}},
\end{equation}
for all $\g \in [0, 3/2)$. Similarly, we also have
\begin{equation}\label{2142}
\|\n B(t)\|_{L^2_\g}
\ge C t^{-\frac{3}{4}+\frac{\g}{2}-\frac{1}{2}},
\end{equation}
for all $\g \in [0, 3/2)$. If the weighed exponent $\g \in [0, 1]$,
it holds on
$$
\|\n^2 B(t)\|_{L^2}
\le C\|\n^2 B(t)\|_{L^2_{\g}}^{\frac{1}{2}} \|\n^2 B(t)\|_{L^2_{-\g}}^{\frac{1}{2}}
\le C\|\n^2 B(t)\|_{L^2_{\g}}^{\frac{1}{2}}\|\n^2 B(t)\|_{\dot{H}^{\g}}^{\frac{1}{2}},
$$
which, together with decay rate \eqref{Decay23}, yields directly
\begin{equation}\label{2143}
\|\n^2 B(t)\|_{L^2_\g}
\ge C t^{-\frac{7}{4}+\frac{\g}{2}},
\end{equation}
and hence, we use the magnetic field equation in \eqref{eq-per2} to obtain
\begin{equation}\label{2144}
\|\p_t B(t)\|_{L^2_\g}\ge Ct^{-\frac{7}{4}+\frac{\g}{2}}.
\end{equation}
The combination of estimates \eqref{2141}-\eqref{2144}
complete the proof of estimates \eqref{Decay45} and \eqref{Decay46}.
Thus, we have already obtained the optimal space-time decay rate for the magnetic field in this paper.

If the Fourier transform
$\mathcal{F}(\vr_0, m_0)=(\widehat{\vr_0}, \widehat{m_0})$
satisfies
\begin{equation*}
|\widehat{\vr_0}|\ge c_0, ~ \widehat{m_0}=0, ~0 \le |\xi| \ll 1,
\end{equation*}
where $c_0$ is a positive constant.
Similar to \eqref{2141}, we use decay rates \eqref{Decay21} and \eqref{Decay22} to obtain
\begin{equation*}
\min\{\|\vr(t)\|_{L^2_\g}, \|u(t)\|_{L^2_\g}\}\ge C t^{-\frac{3}{4}+\frac{\g}{2}}.
\end{equation*}
Thus, it seems that the decay rate
$$
\|(\vr, u)(t)\|_{L^2_\g}=\mathcal{O}(t^{-\frac{3}{4}+\frac{\g}{2}}),
$$
should be the sharp decay rate for the density and velocity.
This will be investigated in future.

\section{Proof of Some Technical Estimates}\label{technical}

In this section, we will establish the claim estimates that have been used in
Sections \ref{lower-bound} and \ref{space-time}.
More precisely, we establish the claim estimates \eqref{claim1}, \eqref{claim2}, \eqref{claim3},
\eqref{Claim-Decay}, \eqref{claim4}, and \eqref{claim5}.
.

\textbf{\underline{Proof of inequality \eqref{claim1}:}}
Multiplying the first and second of \eqref{eq-differ} by $\vr_\d$ and $m_\d$ respectively,
it holds on
\begin{equation*}
\frac{d}{dt}\frac{1}{2}\int (|\vr_\d|^2+|m_\d|^2)dx
+\mu \int |\n m_\d|^2 dx +(\mu+\nu)\int |{\rm div} m_\d|^2 dx=\int S_1 \cdot \n m_\d dx.
\end{equation*}
By virtue of the Taylor expression formula, it holds on
$$
P(1+\vr)-P(1)-\vr \sim \vr^2,
$$
which, together with the Sobolev inequality, yields directly
\begin{equation*}
\|S_1\|_{L^2}\le C\|(\vr, u, B)\|_{H^2}\|\n (\vr, u, B)\|_{L^2}.
\end{equation*}
where the symbol $\sim$ represents the equivalent relation.
Then, we get
\begin{equation}\label{311}
\frac{d}{dt}\|(\vr_\d, m_\d)\|_{L^2}^2
+\mu\|\n m_\d\|_{L^2}^2 +(\mu+\nu)\|\nc m_\d\|_{L^2}^2 dx
\le C\|(\vr, u, B)\|_{H^2}^2\|\n (\vr, u, B)\|_{L^2}^2.
\end{equation}
Applying the equation \eqref{eq-differ}, it is easy to obtain for $k=1,2$,
\begin{equation}\label{312}
\frac{d}{dt}\|\n^k (\vr_\d, m_\d)\|_{L^2}^2
+\mu \|\n^{k+1} m_\d\|_{L^2}^2+(\mu+\nu)\|\n^k \nc m_\d\|_{L^2}^2
\le \|\n^k S_1\|_{L^2} \|\n^{k+1} m_\d\|_{L^2}.
\end{equation}
Now we give the estimates for $\|\n^k S_1\|_{L^2}^2, k=1,2$.
Indeed, we apply the Morse and Sobolev inequalities to obtain
\begin{equation*}
\begin{aligned}
\|\n^k((1+\vr)u\otimes u)\|_{L^2}
&\le C\|1+\vr\|_{L^\infty}\|u\|_{L^\infty}\|\n^k u\|_{L^2}
     +\|u\|_{L^\infty}^2\|\n^k \vr\|_{L^2}\\
&\le C(1+\|\n u\|_{H^1})\|\n u\|_{H^1}\|\n^k (\vr, u)\|_{L^2}\\
&\le C\|\n u\|_{H^1}\|\n^k (\vr, u)\|_{L^2}.
\end{aligned}
\end{equation*}
Similarly, we also have for $k=1,2$,
\begin{equation*}
\|\n^{k+1} (\vr u)\|_{L^2}+\|\n^{k} {\rm div}(\vr u)\|_{L^2}
  \le C\|\n (\vr, u)\|_{H^1}\|\n^{k+1} (\vr, u)\|_{L^2},
\end{equation*}
and
\begin{equation*}
\|\n^k(\frac{1}{2}|B|^2\mathbb{I}_{3\times 3}-B \otimes B)\|_{L^2}
\le C\|\n B\|_{H^1}\|\n^k B\|_{L^2}.
\end{equation*}
By virtue of the Taylor expression formula, we get
$$
P'(1+\vr)\n \vr-P'(1)\n \vr \sim \vr \n \vr,
$$
and
$$
P''(1+\vr)\n \vr \n \vr+P'(1+\vr)\n^2 \vr-P'(1)\n^2 \vr
\sim \n \vr \n \vr+\vr \n^2 \vr.
$$
Then, we use Sobolev inequality to obtain
\begin{equation*}
\|\nabla^k (P(1+\vr)-P(1)-\vr )\|_{L^2}
\le C\|\n \vr\|_{H^1}\|\n^k \vr\|_{L^2}.
\end{equation*}
for $k=1,2$.
Thus, it holds on for $k=1,2$,
\begin{equation}\label{313}
\|\n^k S_1\|_{L^2}\le C\|\n (\vr, u, B)\|_{H^1}\|\n^k (\vr, u, B)\|_{H^1}.
\end{equation}
Then, we use the equation \eqref{312} and Cauchy inequality to get
\begin{equation*}
\frac{d}{dt}\|\n^k (m_\d, \vr_\d)\|_{L^2}^2+\mu\|\n^{k+1}m_\d\|_{L^2}^2
\le C\|\n (\vr, u, B)\|_{H^1}^2 \|\n^k (\vr, u, B)\|_{H^1}^2,
\end{equation*}
where $k=1,2$.
Therefore, we complete the proof of claim estimate \eqref{claim1}.

\textbf{\underline{Proof of inequality \eqref{claim2}:}}
Taking $k(k=0,1)-$th spatial derivative to the second equation of \eqref{eq-differ}
and multiplying the equation by $\n^{k+1} \vr_\d$, then we have
\begin{equation*}
\begin{aligned}
&\int \p_t \n^k m_\d \cdot \n^{k+1} \vr_\d dx+\int |\n^{k+1} \vr_\d|^2 dx\\
&=\int (\mu \n^k \Delta m_\d+(\mu+\nu)\n^{k+1} {\rm div} m_\d)\cdot \n^{k+1} \vr_\d dx
  -\int \n^k {\rm div} S_1\cdot \n^{k+1} \vr_\d dx.
\end{aligned}
\end{equation*}
Using the first equation of \eqref{eq-differ}, it holds on
\begin{equation*}
\begin{aligned}
\int \p_t \n^k m_\d \cdot \n^{k+1} \vr_\d dx
&=\frac{d}{dt}\int \n^k m_\d \cdot \n^{k+1} \vr_\d dx
  -\int \n^k m_\d \cdot \n^{k+1} \p_t \vr_\d dx\\
&=\frac{d}{dt}\int \n^k m_\d \cdot \n^{k+1} \vr_\d dx
  +\int \n^k m_\d \cdot \n^{k+1} {\rm div} m_\d dx\\
&=\frac{d}{dt}\int \n^k m_\d \cdot \n^{k+1} \vr_\d dx
  -\int |\n^k {\rm div} m_\d|^2 dx.
\end{aligned}
\end{equation*}
Thus, we combine the above two equalities to obtain
\begin{equation}\label{321}
\begin{aligned}
&\frac{d}{dt}\int \n^k m_\d \cdot \n^{k+1} \vr_\d dx+\int |\n^{k+1} \vr_\d|^2 dx\\
=&\int |\n^k \nc m_\d|^2 dx-\int \n^k {\rm div} S_1\cdot \n^{k+1} \vr_\d dx\\
 &+\int (\mu \n^k \Delta m_\d+(\mu+\nu)\n^{k+1} {\rm div} m_\d)\cdot \n^{k+1} \vr_\d dx,
\end{aligned}
\end{equation}
which, together with Cauchy inequality, yields directly
\begin{equation*}
\frac{d}{dt}\int \n^k m_\d \cdot \n^{k+1} \vr_\d dx
+\frac{1}{2}\int |\n^{k+1} \vr_\d|^2 dx
\le C(\|\n^{k+1} m_\d\|_{H^1}^2+\|\n^{k+1} S_1\|_{L^2}^2).
\end{equation*}
This and the estimate \eqref{313} implies \eqref{claim2}.
Therefore, we complete proof of claim estimate \eqref{claim2}.

\textbf{\underline{Proof of inequality \eqref{claim3}:}}
Similar to \eqref{312}, it holds on for $k=1,2,3$,
\begin{equation*}
\frac{d}{dt}\frac{1}{2}\int |\n^k B_\d|^2 dx+\int |\n^{k+1} B_\d|^2 dx
\le \|\n^k S_2\|_{L^2}\|\n^{k+1} B_\d\|_{L^2}.
\end{equation*}
Using Morse and Sobolev inequalities, we find
\begin{equation*}
\|\n^k (u \times B)\|_{L^2}\le C\|(u, B)\|_{L^\infty}\|\n^k (u, B)\|_{L^2},
\end{equation*}
and
\begin{equation*}
\begin{split}
\|\n^k ((\nt B)\times B)\|_{L^2}
\le C\|\n B\|_{L^\infty}\|\n^k B\|_{L^2}
      +C\|B\|_{L^\infty}\|\n^{k+1} B\|_{L^2}.
\end{split}
\end{equation*}
Due to the Taylor expression formula, it holds on
$$
\frac{1}{1+\vr}=\frac{1}{1+\vr}-1+1 \sim \vr+1.
$$
Applying the Morse and Sobolev inequalities, we have
\begin{equation*}
\begin{aligned}
&\|\n^k \{\frac{1}{1+\vr}(\nt B)\times B \}\|_{L^2}\\
\le &C\|\vr\|_{L^\infty}\|\n^k ((\nt B)\times B)\|_{L^2}
      +C\|(\nt B)\times B\|_{L^\infty}\|\n^{k} \{\frac{1}{1+\vr}\}\|_{L^2}\\
\le &C\|\n^k ((\nt B)\times B)\|_{L^2}+C\|\n B\|_{L^\infty}\|\n^{k} \vr\|_{L^2}.
\end{aligned}
\end{equation*}
Then, we can obtain the following estimate
\begin{equation*}
\|\n^k (\frac{(\nt B)\times B}{1+\vr})\|_{L^2}
\le C\|\n B\|_{L^\infty}\|\n^{k}(\vr, B)\|_{L^2}
    +C\|B\|_{L^\infty}\|\n^{k+1} B\|_{L^2}.
\end{equation*}
Thus, we can obtain the following estimate
\begin{equation}\label{331}
\frac{d}{dt} \int |\n^k B_\d|^2 dx+\int |\n^{k+1} B_\d|^2 dx
\le C \|(u, B)\|_{W^{1,\infty}}^2\|\n^k (\vr, u, B)\|_{L^2}^2+\|B\|_{L^\infty}^2\|\n^{k+1} B\|_{L^2}^2.
\end{equation}
Therefore, we complete the proof of claim estimate \eqref{claim3}.

\textbf{\underline{Proof of inequality \eqref{Claim-Decay}:}}
Similar to the inequality \eqref{331}, we can obtain
\begin{equation*}
\frac{d}{dt}\|\n^3 B\|_{L^2}^2 +\|\n^4 B\|_{L^2}^2
\le C \|(u, B)\|_{W^{1,\infty}}^2\|\n^3 (\vr, u, B)\|_{L^2}^2+C\|B\|_{L^\infty}^2\|\n^4 B\|_{L^2}^2.
\end{equation*}
which, together with the smallness of initial data, implies directly
\begin{equation*}
\frac{d}{dt}\|\n^3 B\|_{L^2}^2 +\|\n^4 B\|_{L^2}^2
\le C \|(u, B)\|_{W^{1,\infty}}^2\|\n^3 (\vr, u, B)\|_{L^2}^2.
\end{equation*}
Multiplying the above inequality by $(1+t)^3$, then we have
\begin{equation*}
\begin{aligned}
&\frac{d}{dt}[(1+t)^3\|\n^3 B\|_{L^2}^2]+(1+t)^3\|\n^4 B\|_{L^2}^2\\
&\le 3(1+t)^2\|\n^3 B\|_{L^2}^2+C(1+t)^3\|(u, B)\|_{W^{1,\infty}}^2\|\n^3 (\vr, u, B)\|_{L^2}^2\\
&\le C(1+t)^{-\frac{5}{2}}.
\end{aligned}
\end{equation*}
Integrating over the above inequality over $[T_*, t]$ and using the uniform estimate \eqref{uniform}, we have
$$
\int_{T_*}^t (1+\tau)^3\|\n^4 B\|_{L^2}^2 d\tau \le C,
$$
where $C$ is a positive constant independent of time. Therefore we complete the proof of lemma.

\textbf{\underline{Proof of inequality \eqref{claim4}:}}
Applying $\n^k(k=1,2,3)$ operator to the third equation in \eqref{eq-per2},
multiplying by $\x^{2\g} \n^k B$, and integrating by part, it is easy to check
\begin{equation}\label{351}
\begin{aligned}
&\frac{d}{dt} \|\n^k B\|_{L^2_\g}^2+\|\n^{k+1} B\|_{L^2_\g}^2\\
\le
&C\|\n^k B\|_{L^2_{\g-1}}^2
    +C(\|\x^{2\g} |\n^{k-1} G_3| |\n^{k+1}B|\|_{L^1}+\|\x^{2\g-1} |\n^{k-1} G_3| |\n^{k}B|\|_{L^1}).
\end{aligned}
\end{equation}
For $k=1,2$, it is easy to check that
\begin{equation}\label{352}
\begin{aligned}
\frac{d}{dt}\|\n B(t)\|_{L^2_\g}^2+\|\n^2 B\|_{L^2_\g}^2
\le
&\|\n B(t)\|_{L^2_{\g-1}}^2+C\|u\|_{L^\infty}^2 \|\n B\|_{L^2_\g}^2,\\
&+C\|\n^2 (\vr, u, B)\|_{H^1}^2(\|\n B\|_{L^2_\g}^2+\|B\|_{L^2_{\g-1}}^2),
\end{aligned}
\end{equation}
and
\begin{equation}\label{353}
\begin{aligned}
\frac{d}{dt}\|\n^2 B(t)\|_{L^2_\g}^2+\|\n^3 B\|_{L^2_\g}^2
\le
&C\|\n^2 B(t)\|_{L^2_{\g-1}}^2+C\|u\|_{L^\infty}^2 \|\n^2 B\|_{L^2_\g}^2\\
& +C\|\n^2 (\vr, u, B)\|_{H^1}^2(\|\n B\|_{H^1_\g}^2+\|B\|_{H^1_{\g-1}}^2).
\end{aligned}
\end{equation}
Now we will deal with the case $k=3$ in equality \eqref{351} in detail.
Taking $k=3$ in \eqref{351}, it holds on
\begin{equation}\label{354}
\begin{aligned}
&\frac{d}{dt} \|\n^3 B\|_{L^2_\g}^2+\|\n^4 B\|_{L^2_\g}^2\\
\le
&C\|\n^3 B\|_{L^2_{\g-1}}^2
    +C(\|\x^{2\g} |\n^2 G_3| |\n^4 B|\|_{L^1}+\|\x^{2\g-1} |\n^2 G_3| |\n^3 B|\|_{L^1}).
\end{aligned}
\end{equation}
Using the H\"{o}lder, Cauchy and Sobolev inequalities, we have
\begin{equation}\label{355}
\begin{aligned}
&\|\x^{2\g} |\n^3 (u \times B)| |\n^4 B|\|_{L^1}\\
\le
&C(\|u\|_{L^\infty} \| \n^3 B\|_{L^2_\g}
     +\|\n u\|_{L^\infty} \|\n^2 B\|_{L^2_\g})\|\n^4 B\|_{L^2_\g}\\
&+C(\|\n^2 u\|_{L^3} \|\n B\|_{L^6_\g}
     +\|\n^3 u\|_{L^2} \|B\|_{L^\infty_\g})\|\n^4 B\|_{L^2_\g}\\
\le
&\var \|\n^4 B\|_{L^2_\g}^2
  +C\|u\|_{L^\infty}^2 \| \n^3 B\|_{L^2_\g}^2
  +C\|\n^2 u\|_{H^1}^2(\|\n^2 B\|_{L^2_\g}^2+\|\n B\|_{L^6_\g}^2+\|B\|_{L^\infty_\g}^2).
\end{aligned}
\end{equation}
Using the Sobolev inequality, it holds on
\begin{equation}\label{356}
\|B\|_{L^\infty_\g}^2\le C\|\n (\x^\g B)\|_{H^1}^2
\le C(\|\n B\|_{H^1_\g}^2+\|B\|_{H^1_{\g-1}}^2+\|B\|_{L^2_{\g-2}}^2),
\end{equation}
and
\begin{equation}\label{357}
\|\n B\|_{L^6_\g}^2\le \|\n (\x^\g \n B)\|_{L^2}^2\le C(\|\n^2 B\|_{L^2_\g}^2+\|\n B\|_{L^2_{\g-1}}^2).
\end{equation}
Substituting estimates \eqref{356} and \eqref{357} into \eqref{355}, one arrives at
\bq\label{358}
\begin{aligned}
&\|\x^{2\g} |\n^3 (u \times B)| |\n^4 B|\|_{L^1}\\
\le
&\var \|\n^4 B\|_{L^2_\g}^2
  +C\|u\|_{L^\infty}^2 \| \n^3 B\|_{L^2_\g}^2
  +C\|\n^2 u\|_{H^1}^2(\|\n B\|_{H^2_\g}^2+\|B\|_{H^2_{\g-1}}^2+\|B\|_{L^2_{\g-2}}^2).
\end{aligned}
\eq
Similarly, it is easy to check that
\bq\label{359}
\begin{aligned}
&\|\x^{2\g-1} |\n^3 (u \times B)| |\n^3 B|\|_{L^1}\\
\le
&\var \|\n^3 B\|_{L^2_{\g-1}}^2
  +C\|u\|_{L^\infty}^2 \| \n^3 B\|_{L^2_\g}^2
  +C\|\n^2 u\|_{H^1}^2(\|\n B\|_{H^2_\g}^2+\|B\|_{H^2_{\g-1}}^2+\|B\|_{L^2_{\g-2}}^2),\\
\end{aligned}
\eq
and
\bq\label{3510}
\begin{aligned}
&\|\x^{2\g} |\n^3 [\frac{(\nt B)\times B }{1+\vr}]| |\n^4 B|\|_{L^1}
  +\|\x^{2\g-1} |\n^3 [\frac{(\nt B)\times B }{1+\vr}]| |\n^3 B|\|_{L^1}\\
\le
&\var (\|\n^4 B\|_{L^2_\g}^2+\|\n^3 B\|_{L^2_{\g-1}}^2)
  +C\|\n^2 (\vr, B)\|_{H^1}^2(\|\n B\|_{H^2_\g}^2+\|B\|_{H^2_{\g-1}}^2+\|B\|_{L^2_{\g-2}}^2).
\end{aligned}
\eq
Since $G_3=\nabla\times (u \times B)-\nabla\times (\frac{(\nt B)\times B }{1+\vr})$, we plug the estimates
\eqref{358}, \eqref{359} and \eqref{3510} into \eqref{357} to get
\begin{equation*}
\begin{aligned}
&\frac{d}{dt}\|\n^3 B(t)\|_{L^2_\g}^2+\|\n^4 B\|_{L^2_\g}^2\\
\le
&C\|\n^3 B(t)\|_{L^2_{\g-1}}^2
    +C\|u\|_{L^\infty}^2 \|\n^3 B\|_{L^2_\g}^2\\
&  +C\|\n^2 (\vr, u, B)\|_{H^1}^2(\|\n B\|_{H^2_\g}^2+\|B\|_{H^2_{\g-1}}^2+\|B\|_{L^2_{\g-2}}^2),
\end{aligned}
\end{equation*}
which, together with \eqref{352} and \eqref{353}, complete the proof of claim inequality \eqref{claim4}.

\textbf{\underline{Proof of inequality \eqref{claim5}:}}
Step 1: Applying $\n^k(k=1,2,3)$ operator to \eqref{eq-per2}$_1$ and \eqref{eq-per2}$_2$,
and  multiplying by $\x^{2\g} \n^k \vr$ and $\x^{2\g} \n^k u$ respectively, we have
\begin{equation}\label{361}
\begin{aligned}
&\frac{d}{dt}\frac{1}{2}\int \x^{2\g} (|\n^k \vr|^2+|\n^k u|^2) dx
+\mu \int \x^{2\g}|\n^{k+1} u|^2 dx
+(\mu+\nu)\int \x^{2\g}|\n^k\nc u|^2 dx\\
=
&\int \n^{k} G_1 \cdot \x^{2\g} \n^k \vr dx
 -\int \n^{k-1} G_2 \cdot \n (\x^{2\g} \n^k u) dx+II^k_1+II^k_2.
\end{aligned}
\end{equation}
where the functions $II^k_1$ and $II^k_2$ are defined by
$$
II^k_1= 2 \g \int \x^{2\g-2}x_j \p_i x_j \n^k \vr \cdot \n^k u_i dx,
$$
and
$$
II^k_2=-2\g (\mu+\nu)\int \x^{2\g-2}x_k \p_i x_k \n^k u_i \n^k \p_j u_j dx
    -2\g \mu \int \x^{2\g-2}x_k \p_j x_k \n^k \p_j u_i \n^k u_i dx.
$$
For the case $k=1,2$, it is easy to check that
\begin{equation}\label{362}
\begin{aligned}
&\frac{d}{dt}(\|\n^k \vr\|_{L^2_\g}^2+\|\n^k u\|_{L^2_\g}^2)+\mu \|\n^{k+1} u\|_{L^2_\g}^2\\
\le
& \var \|\n^{k+1} \vr\|_{L_\gamma^2}^2+C\|\n^k \vr\|_{L^2_{\g-1}}^2+C\|\n^k u\|_{L^2_{\g-1}}^2
  +C\|\n^k \vr\|_{L^2_{\g}}\|\n^k u\|_{L^2_{\g-1}}\\
&+C\|\n (\vr, u, B)\|_{H^{k}}^2 \|\n (\vr, u, B)\|_{H^{k-1}_\g}^2.
\end{aligned}
\end{equation}
Now, we deal with the case $k=3$ in \eqref{361}.
Using H\"{o}lder and Cauchy inequalities, we obtain
\begin{equation}\label{363}
|II^3_1|
\le \|\n^3 \vr\|_{L^2_{\g}}\|\n^3 u\|_{L^2_{\g-1}}.
\end{equation}
and
\begin{equation}\label{364}
\begin{aligned}
|II^3_2|
&\le C\mu \|\n^4 u\|_{L^2_\g}\|\n^3 u\|_{L^2_{\g-1}}
     +C(\mu+\nu) \|\n^3 \nc u\|_{L^2_\g}\|\n^3 u\|_{L^2_{\g-1}} \\
&\le \var \mu \|\n^4 u\|_{L^2_\g}^2+\var (\mu+\nu)\|\n^3 \nc u\|_{L^2_\g}^2
         +C\|\n^3 u\|_{L^2_{\g-1}}^2.
\end{aligned}
\end{equation}
Now, we deal with the term $\int \n^{3} G_1 \cdot \x^{2\g} \n^3 \vr dx$.
Obviously, it holds on
\begin{equation*}
\int \n^{3} G_1 \cdot \x^{2\g} \n^3 \vr dx
=-\int (u\cdot \n\n^3 \vr)\cdot \x^{2\g}\n^3 \vr dx
 +\int (\n^{3} G_1+u\cdot \n \n^3 \vr) \cdot \x^{2\g} \n^3 \vr dx.
\end{equation*}
Indeed, integrating by part, one arrives at
\begin{equation*}
\int (u\cdot \n\n^3 \vr)\cdot \x^{2\g}\n^3 \vr dx
=-\frac{1}{2}\int |\n^3 \vr|^2 {\rm div} (\x^{2\g}u) dx,
\end{equation*}
which, using the H\"{o}lder and Cauchy inequalities, yields directly
\begin{equation}\label{365}
\begin{aligned}
|\int |\n^3 \vr|^2 {\rm div} (\x^{2\g}u) dx|
\le
&C\|\n u\|_{L^\infty}\|\n^3 \vr\|_{L^2_\g}^2
    +C\|u\|_{L^\infty}\|\n^3 \vr\|_{L^2_\g} \|\n^3 \vr\|_{L^2_{\g-1}}\\
\le
&\d_0 \|\n^3 \vr\|_{L^2_\g}^2+C\|\n^3 \vr\|_{L^2_{\g-1}}^2.
\end{aligned}
\end{equation}
Using the H\"{o}lder, Sobolev and Cauchy inequalities, it holds on
\begin{equation}\label{366}
\begin{aligned}
&|\int (\n^{3} G_1+u\cdot \n \n^3 \vr) \cdot \x^{2\g} \n^3 \vr dx|\\
\le
&C\|\n u\|_{L^\infty}\|\n^3 \vr\|_{L^2_\g}^2
  +C\|\n^2 u\|_{L^3}\|\n^2 \vr\|_{L^6_\g}\|\n^3 \vr\|_{L^2_\g}\\
& +C\|\n \vr \|_{L^\infty}\|\n^3 u\|_{L^2_\g}\|\n^3 \vr\|_{L^2_\g}
  +C\|\vr \|_{L^\infty}\|\n^4 u\|_{L^2_\g}\|\n^3 \vr\|_{L^2_\g}\\
\le
&\var \|\n^4 u\|_{L^2_\g}^2+\d_0\|\n^3 \vr\|_{L^2_\g}^2
  +C\|\n^2 \vr\|_{L^2_{\g-1}}^2+C\|\n^2 \vr\|_{H^1}^2\|\n^3 u\|_{L^2_\g}^2,
\end{aligned}
\end{equation}
where we have used the uniform bound \eqref{uniform} and the following inequality
$$
\|\n^2 \vr\|_{L^6_\g}=\|\x^\g \n^2 \vr\|_{L^6}
\le C(\|\n^3 \vr\|_{L^2_\g}+\|\n^2 \vr\|_{L^2_{\g-1}}).
$$
Combining the estimates \eqref{365} and \eqref{366}, we obtain
\begin{equation}\label{367}
|\int \n^{3} G_1 \cdot \x^{2\g} \n^3 \vr dx|
\le
\var \|\n^4 u\|_{L^2_\g}^2+\d_0\|\n^3 \vr\|_{L^2_\g}^2
  +C\|\n^2 \vr\|_{H^1_{\g-1}}^2+C\|\n^2 \vr\|_{H^1}^2\|\n^3 u\|_{L^2_\g}^2.
\end{equation}
Similarly, we can get
\begin{equation}\label{368}
\begin{aligned}
&|\int \n^2 G_2 \cdot \n (\x^{2\g} \n^3 u) dx|\\
\le
& \var \|\n^4 u\|_{L^2_\g}^2+C\|\n^3 u\|_{L^2_{\g-1}}^2+C\|\n^2 G_2\|_{L^2_\g}^2\\
\le
&(\var+\d_0)\|\n^4 u\|_{L^2_\g}^2+C\|\n^2 u\|_{H^1_{\g-1}}^2
 +C\|\n (\vr, u, B)\|_{H^2}^2\|\n (\vr, u, B)\|_{H^2_\g}^2.
\end{aligned}
\end{equation}
Combining the estimates \eqref{363}, \eqref{364}, \eqref{367} and \eqref{368}, it holds on
\begin{equation*}
\begin{aligned}
&\frac{d}{dt}(\|\n^3 \vr\|_{L^2_\g}^2+\|\n^3 u\|_{L^2_\g}^2)+\mu \|\n^4 u\|_{L^2_\g}^2\\
\le
& \d_0 \|\n^3 \vr\|_{L^2}^2+C\|\n^2 \vr\|_{H^1_{\g-1}}^2+C\|\n^2 u\|_{H^1_{\g-1}}^2
  +C\|\n^3 \vr\|_{L^2_{\g}}\|\n^3  u\|_{L^2_{\g}}\\
&+C\|\n (\vr, u, B)\|_{H^{2}}^2 \|\n (\vr, u, B)\|_{H^2_\g}^2,
\end{aligned}
\end{equation*}
which, together with estimates \eqref{296} and \eqref{362}, yields directly
\begin{equation}\label{369}
\begin{aligned}
&\frac{d}{dt}(\|\n^k \vr (t)\|_{H^{3-k}_\g}^2+\|\n^k u \|_{H^{3-k}_\g}^2)
  +\mu\|\n^{k+1} u\|_{H^{3-k}_\g}^2\\
\le
&    C(\var+\d_0)\|\n^{k+1} \vr \|_{H^{2-k}_\g}^2
     +C\|\n^k \vr\|_{H^{3-k}_{\g-1}}^2+C\|\n^k u\|_{H^{3-k}_{\g-1}}^2\\
&    +C\|\n^k \vr\|_{H^{3-k}_\g}\|\n^k u\|_{H^{3-k}_{\g-1}}
     +C\|\n (\vr, u, B)\|_{H^2}^2 \|\n (\vr, u, B)\|_{H^2_\g}^2,
\end{aligned}
\end{equation}
where $k=0,1,2$.

Step 2: we establish the estimate for the dissipation of density.
Indeed, similar to the estimate \eqref{321}, we have for all $k=0,1,2$,
\begin{equation*}
\begin{aligned}
&\frac{d}{dt}\int \n^k u \cdot \x^{2\g} \n^{k+1}\vr dx
  +\int \x^{2\g}|\n^{k+1}\vr|^2 dx\\
=
&\int [\n^k G_2+\mu \n^k \Delta u+(\mu+\nu)\n^{k+1}{\rm div}u]\cdot \x^{2\g} \n^{k+1}\vr dx
 +\int \n^k u \cdot \x^{2\g}\n^{k+1}\vr_t dx.
\end{aligned}
\end{equation*}
Integrating by part and using the density equation, we obtain
\begin{equation*}
\begin{aligned}
&\int \n^k u \cdot \x^{2\g}\n^{k+1}\vr_t dx\\
=&-\int \x^{2\g}\n^{k}{\rm div}u \cdot \n^{k}\vr_t dx
  -2\g \int  \x^{2\g-2} x_i \p_j x_i  \n^k u_j \cdot \n^{k}\vr_t dx\\
=&\int \x^{2\g} |\n^{k}{\rm div} u|^2 dx
  +2\g \int  \x^{2\g-2} x_i \p_j x_i  \n^k u_j \cdot \n^{k} {\rm div}u dx\\
& -\int \x^{2\g}\n^{k}{\rm div} u \cdot \n^{k} G_1 dx
  -2\g \int  \x^{2\g-2} x_i \p_j x_i  \n^k u_j \cdot \n^{k} G_1 dx.
\end{aligned}
\end{equation*}
And then, we apply the H\"{o}lder and Cauchy inequalities to get
\begin{equation*}
\begin{aligned}
&\frac{d}{dt}\int \n^k u \cdot \x^{2\g} \n^{k+1}\vr dx
  +\frac{1}{2}\int \x^{2\g}|\n^{k+1}\vr|^2 dx\\
&\le C\|\n^{k+1} u\|_{H^1_\g}^2
     +C\|\n^k u\|_{L^2_{\g-1}}^2
     +C(\|\n^k G_1\|_{L^2_\g}^2+\|\n^k G_2\|_{L^2_\g}^2).
\end{aligned}
\end{equation*}
Then, similar to \eqref{363} and \eqref{369}, it is easy to check that
\begin{equation}\label{3610}
\begin{aligned}
&\frac{d}{dt}\sum_{k\le l \le 2}\int \n^l u \cdot \x^{2\g} \n^{l+1}\vr dx
  +\frac{1}{2}\|\n^{k+1}\vr \|_{H^{2-k}_\g}^2 \\
&\le C\|\n^{k+1} u\|_{H^{3-k}_\g}^2+C\|\n^k u\|_{H^{3-k}_{\g-1}}^2
     +C\|\n (\vr, u, B)\|_{H^2}^2 \|\n (\vr, u, B)\|_{H^2_\g}^2,
\end{aligned}
\end{equation}
where $k=0,1,2$. Then, multiplying \eqref{3610} by $4C\d_0$
and adding with \eqref{369}, we obtain the claim inequality \eqref{claim5}.
Therefore, we complete the proof of the claim inequality \eqref{claim5}.

\section*{Acknowledgements}
Jincheng Gao's research was partially supported by
Fundamental Research Funds for the Central Universities(Grants No.18lgpy66) and NNSF of China(Grants No.11801586).
Zheng-an Yao's research was partially supported by NNSF of China(Grant Nos.11431015 and 11971496).

\phantomsection

\end{document}